\documentclass[11pt]{amsart}
\usepackage{amssymb,amsmath,amsthm}
\usepackage{graphicx}
\usepackage{rotating}
\usepackage{arydshln}

\usepackage{amssymb,amsmath,amsthm}
\usepackage[all]{xy}
\usepackage{xcolor,colortbl}
\usepackage{soul}

\usepackage[all]{xy}
\usepackage{xcolor,colortbl}
\usepackage{soul}

\newtheorem{theorem}{Theorem}[section]
\newtheorem{proposition}[theorem]{Proposition}
\newtheorem{lemma}[theorem]{Lemma}
\newtheorem{corollary}[theorem]{Corollary}
 
\theoremstyle{definition}
\newtheorem{definition}[theorem]{Definition}
\newtheorem{convention}[theorem]{Convention}
\newtheorem{example}[theorem]{Example}
\newtheorem{remark}[theorem]{Remark}

\numberwithin{equation}{section}

\newcommand{\sfD}{\mathsf{D}}
\newcommand{\sfE}{\mathsf{E}}

\newcommand{\sfI}{\mathsf{I}}

\newcommand{\bbZ}{{\mathbb{Z}}}
\newcommand{\bbP}{{\mathbb{P}}}
\newcommand{\bbG}{{\mathbb{G}}}
\newcommand{\bbC}{{\mathbb{C}}}
\newcommand{\bbQ}{{\mathbb{Q}}}

\newcommand{\bbk}{{\Bbbk}}

\newcommand{\Aut}{\operatorname{Aut}}

\newcommand{\Spec}{\operatorname{Spec}}
\newcommand{\id}{\operatorname{id}}
\newcommand{\Pic}{\operatorname{Pic}}

\newcommand{\tr}{\operatorname{tr}}
\newcommand{\Ker}{\operatorname{Ker}}

\newcommand{\im}{\operatorname{im}}

\newcommand{\Num}{\operatorname{Num}}
\newcommand{\Or}{\operatorname{O}}
\newcommand{\nt}{\operatorname{nt}}
\newcommand{\ct}{\operatorname{ct}}
\newcommand{\Tors}{\operatorname{Tors}}

\newcommand{\MW}{\operatorname{MW}}

\newcommand{\ord}{\operatorname{ord}}
\newcommand{\Bs}{\operatorname{Bs}}

\newcommand{\til}{\operatorname{\tilde{\ell}}}

\newcommand{\Disc}{\operatorname{Disc}}

\newcommand{\la}{\langle}
\newcommand{\ra}{\rangle}

\newcommand{\half}{\frac{1}{2}}
\newcommand{\calR}{\mathcal{R}}
\newcommand{\calO}{\mathcal{O}}

\newcommand{\calD}{\mathcal{D}}

\newcommand{\calE}{\mathcal{E}}
\newcommand{\calL}{\mathcal{L}}
\newcommand{\calP}{\mathcal{P}}
\newcommand{\calQ}{\mathcal{Q}}

\newcommand{\frakf}{\mathfrak{f}}

\newcommand{\et}{\operatorname{et}}
\newcommand{\beq}{\begin{equation}}
\newcommand{\eeq}{\end{equation}}
\usepackage[
            pdftex,
           colorlinks=true,
            urlcolor=blue,       
            filecolor=blue,     
            linkcolor=blue,       
            citecolor=blue,         
            pdftitle= {Title},
            pdfauthor={Author},
            pdfsubject={Subject},
            pdfkeywords={Key}
            pagebackref,
            bookmarksopen=true]
            {hyperref}
\usepackage{hyperref}

\date{}

\title[Automorphism groups of elliptic surfaces]{Automorphism groups of rational elliptic and quasi-elliptic surfaces in all characteristics}
\author{Igor Dolgachev}
\address{\hfill \newline 
Department of Mathematics \newline
University of Michigan \newline
525 East University Avenue \newline
Ann Arbor,
 MI 48109-1109 USA}
\email{idolga@umich.edu}

\author{Gebhard Martin}
\address{\hfill \newline
Mathematisches Institut  \newline
Universit\"at Bonn \newline
Endenicher Allee 60 \newline
53115 Bonn \newline
Germany}
\email{gmartin@math.uni-bonn.de}

\begin{document}

\begin{abstract} We study the groups of automorphisms of rational algebraic surfaces that admit a relatively minimal pencil of curves of arithmetic genus one over an algebraically closed field of arbitrary characteristic. In particular, we classify such surfaces that admit non-trivial automorphisms that act trivially on the Picard group. As an application, we classify classical Enriques surfaces in characteristic $2$ that admit non-trivial numerically trivial automorphisms.
\end{abstract}

\maketitle

\setcounter{tocdepth}{1}

\tableofcontents

 \section{Introduction}
Let $S$ be a rational surface over an algebraically closed field $\Bbbk$ that admits a relatively minimal fibration over $\bbP^1$ whose generic fiber $F_\eta$, i.e. the fiber  over the generic point $\eta$ of $\bbP^1$, is an irreducible curve of arithmetic genus one over the field $K = \Bbbk(\eta)$ of rational functions on $\bbP^1$. We say that $f:S\to \bbP^1$ is an elliptic fibration if $F_\eta$ is smooth, and a quasi-elliptic fibration otherwise. The latter case can happen only if $\Bbbk$ is of characteristic $2$ or $3$. To combine both cases we call $f$ a genus one fibration and call a surface $S$ equipped with such a fibration a genus one surface, or an elliptic or quasi-elliptic surface, if we want to specify the type of the fibration. We recall that the \emph{index} of $S$ is the greatest common divisor of the multiplicities of fibers of $f$, and that $S$ admits a section if and only if it has index $1$. In this paper, we study the group $\Aut(S)$ of biregular automorphisms of $S$, by decomposing $\Aut(S)$ into subgroups that are easier to understand. Moreover, we give a classification of rational genus one surfaces and of classical Enriques surfaces in characteristic $2$ which admit numerically trivial automorphisms, complementing the results of \cite{DolgachevMartin}.

\subsection{Subgroups of $\Aut(S)$}
It is well known that the jacobian variety of $F_\eta$ admits a realization as 
the generic fiber of a genus one fibration $J(f):J(S) \to \bbP^1$ on a rational 
surface $J(S)$ that admits a section. 
After fixing a section $E_0$, the set of sections $J(\bbP^1)$ acquires t
he structure of a finitely generated abelian group with $E_0$ as the zero element. 
This group is called the \emph{Mordell-Weil group} of the fibration and will be 
denoted by $\MW(J(f))$. Via the restriction of sections to the generic fiber, 
the Mordell-Weil group becomes isomorphic to the group of rational points $J_\eta(K)$ of the generic fiber. The Mordell-Weil group acts as translations on $F_\eta$ and, since $f$ is relatively minimal, this action extends to a regular action on $S$.
Now, let
$$\beta:\Aut(S)\to \Aut(\bbP^1).$$
be the natural homomorphism describing the action of $\Aut(S)$ on the base of the fibration and let
$$\rho:\Aut(S) \to \Or(\Pic(S)), \quad g\mapsto g^*$$
be the natural homomorphism describing the action on the Picard group of the surface.

We define the following subgroups of $\Aut(S)$.

\begin{itemize}
\item $\Aut_{\ct}(S):=\Ker(\rho)$.
\item $\Aut(S)^\dagger := \Ker(\beta)$.
\item $\Aut_{\ct}(S)^\dagger := \Aut_{\ct}(S)\cap \Aut(S)^\dagger$.
\item $\Aut_{\tr}(S):= \MW(J(f))$.
\item $\Aut_E(S):= \{g\in \Aut(S): g(E) = E\}$, where $E$ is a $(-1)$-curve on $S$.
\end{itemize}

\subsubsection{Action on the jacobian fibration} 
By Theorem \ref{thm: actiononjacobian}, and since $f$ is the only genus one fibration of $S$, the group $\Aut(S)$ acts on $J(S)$ through a morphism $\varphi: \Aut(S) \to \Aut_{0}(J(S))$, where $\Aut_{0}(J(S))$ is the subgroup of $\Aut(J(S))$ preserving a fixed section $E_0$. The following theorem is a consequence of Theorem \ref{thm: actiononjacobian} and Theorem \ref{delPezzo}.

\begin{theorem}
Let $f: S \to \bbP^1$ be a rational genus one surface of index $m$. Then, there is a homomorphism $\varphi: \Aut(S) \to \Aut_{0}(J(S))$ such that:
\begin{enumerate}
\item $\Ker(\varphi) = \Aut_{\tr}(S)$.
\item $\Aut_{0}(J(S)) = \Aut(\calD)$ for some weak del Pezzo surface $\calD$ of degree $1$.
\item $\im(\varphi) \subseteq \Aut_{0}(J(S))$ is contained in the stabilizer of an $m$-torsion point on the identity component of a fiber of $J(f)$.
\item If $m = 1$, then $\varphi$ admits a section. In particular, $\Aut(S) = \MW(J(f)) \rtimes \Aut(\calD)$.
\end{enumerate}
\end{theorem}

\begin{remark}
The groups $\Aut_{\tr}(S)$ and $\Aut(\calD)$ are well-studied:
\begin{enumerate}
\item In the case of elliptic (resp. quasi-elliptic) surfaces all possible groups  $\Aut_{\tr}(S)$ are listed  in \cite[Table 8.2]{ShiodaSchutt} (resp. \cite{Ito1}, \cite{Ito2}). 
\item The list of all possible finite groups acting on a del Pezzo surface of degree 1 was essentially known since the 19th century. It is contained in \cite{DolgIsk} and \cite[8.8.4]{CAG} (under an assumption on the characteristic).
\item In Section \ref{sec: auttr}, we give a formula for the action of $d \cdot \Aut_{\tr}(S)$ on $\Pic(S)$, where $d$ is the smallest positive integer such that all discriminant groups of fibers of $f$ are $d$-torsion. This generalizes Gizatullin's formula for the case $d = 1$ \cite{Gizatullin}.
\end{enumerate}
\end{remark}

\subsubsection{Action on the base curve}
The following decomposition is well-known and we mention it here for the sake of completeness.

\begin{proposition} \label{prop: actiononbasecurve}
Let $f: S \to \bbP^1$ be a rational genus one surface of index $m$. Then, there is a morphism $\beta: \Aut(S) \to {\rm PGL}_2$ such that:
\begin{enumerate}
\item $\Aut(S)^\dagger := \Ker(\beta) = \Aut_{F_\eta}$.
\item $\im(\beta) \subseteq {\rm PGL}_2$ preserves the image of the set of singular fibers of $f$.
\item The $j$-map of $f$ is $\im(\beta)$-invariant.
\end{enumerate}
\end{proposition} 

\begin{remark}
Properties (2) and (3) of Proposition \ref{prop: actiononbasecurve} could, in principle, be used to classify the groups $\im(\beta)$, but the list would be too long and we decided not to pursue  this computation. Some of these groups appear in Table \ref{Table2}.
\end{remark}
 
\subsubsection{Action on the Picard lattice}


It is a basic fact that the Picard group $\Pic(S)$ of a rational genus one surface $S$ is a free abelian group and the intersection product equips it with the structure of an odd unimodular hyperbolic lattice. A choice of blowing down morphism $\pi:S\to \bbP^2$ allows one to construct an orthonormal basis $(e_0,e_1,\ldots,e_9)$, called geometric basis,  that defines an isomorphism $\phi: \Pic(S)\to \sfI^{1,9}$, called a \emph{marking} of $S$. The image of the orthogonal complement of the canonical class $K_S$ of $S$ is a sublattice of $\sfI^{1,9}$ isomorphic to the root lattice $\sfE_9$ of the affine root system of type $\tilde{E}_8$. In this way, a marking $\phi$ and the homomorphism $\rho$ define a homomorphism 
$$\Phi:\Aut(S)\to W(\sfE_9),$$
where $W(\sfE_9)$ is the Weyl group  of the lattice $\sfE_9$, the index 2 subgroup of $\Or(\sfE_9)$ generated by reflections in roots. 
The lattice $\sfE_9$ has radical of rank 1 generated by a primitive isotropic vector $\frakf$, which is the image of the canonical class $K_S$ under any marking. It is known that the orthogonal complement $\frakf^\perp$ modulo $\bbZ \frakf$ is isomorphic to the unimodular even negative lattice $\sfE_8$. This defines a homomorphism $r:W(\sfE_9)\to W(\sfE_8)$.  Moreover, there is an inclusion $\sfE_8 \hookrightarrow W(\sfE_9)$ defined by the formula $v\mapsto \iota_v$, where $\iota_v(x)=  x+ \langle x,v \rangle \frakf$. 

The following theorem describes the action of $\Aut(S)$ on $\Pic(S)$. It is a consequence of Section \ref{sec: aute} and Corollary \ref{cor: gizatullin}.

\begin{theorem}
Let $f: S \to \bbP^1$ be a rational genus one surface of index $m$. Then, there is a morphism $\Phi: \Aut(S) \to W(\sfE_9)$ such that:
\begin{enumerate}
\item $\Ker(\Phi) = \Aut_{\ct}(S)$.
\item A choice of a $(-1)$-curve $E \subseteq S$ defines a splitting $W(\sfE_9) \cong \sfE_8 \rtimes W(\sfE_8)$.
\item For the splitting in (2) defined by $E$, we have $\Phi(\Aut_E(S)) \subseteq W(\sfE_8)$.
\item $\Phi(d \cdot \Aut_{\tr}(S)) \subseteq E_8$, where $d$ is the smallest positive integer such that all the discriminant groups of fibers of $f$ are $d$-torsion groups.
\end{enumerate}
\end{theorem}

\subsection{Numerically trivial automorphisms of cohomologically flat genus one surfaces with rational jacobians}

The bulk of the paper is devoted to the classification of all rational genus one surfaces with non-trivial group $\Aut_{\ct}(S)$. 
Thanks to the availability of the Weierstrass equation, it is probably not very hard to do it in the case when the fibration coincides with its jacobian fibration. However, we take another, more geometric approach to the classification in this case. The case of non-jacobian fibrations turns out to be much more technically involved. In Theorem \ref{thm: actiononjacobian}, we describe an explicit relationship between the automorphism group of a cohomologically flat genus one surface and the automorphism group of its jacobian.
This does not only allow us to describe non-jacobian rational genus one surfaces with cohomologically trivial automorphisms, but it also allows us  to give a classification of Enriques surfaces over a field of characteristic $2$ with non-trivial canonical class that admit non-trivial numerically trivial automorphisms, i.e., automorphisms acting trivially on the group of divisor classes up to numerical equivalence. This classification was absent in \cite{DolgachevMartin}. 

\subsubsection{Rational genus one surfaces with cohomologically trivial automorphisms}
Recall that $\Aut_{\ct}(S)$ is the subgroup of $\Aut(S)$ that acts trivially on $\Pic(S)$.

\begin{theorem}
Let $f: S \to \bbP^1$ be a rational genus one surface of index $m$. 
\begin{enumerate}
\item If $m = 1$, then $\Aut_{\ct}(S)$ is non-trivial if and only if $S$ is one of the surfaces in Table \ref{Table2}.
\item If $m > 1$ and $\Aut_{\ct}(S)$ is non-trivial, then $S$ occurs in Table \ref{nonjacobiantable}.
\end{enumerate}
\end{theorem}

\subsubsection{Classical Enriques surfaces in characteristic $2$ with numerically trivial automorphisms}

An Enriques surface $S$ is called \emph{classical} if $\omega_S$ is non-trivial. In contrast to the case of rational surfaces, the Picard group $\Pic(S)$ is not torsion-free. We let $\Aut_{\nt}(S) := \Ker( \Aut(S) \to \Pic(S)/{\rm Torsion})$ be the group of numerically trivial automorphisms. In characteristic different from $2$, the classification of Enriques surfaces with numerically trivial automorphisms is known (see \cite{DolgachevMartin} and \cite[Section 8.2]{DK}). We give a classification of classical Enriques surfaces with numerically trivial automorphisms in characteristic $2$. 

\begin{theorem}
Let $S$ be a classical Enriques surface in characteristic $2$. Then, $\Aut_{\nt}(S)$ is non-trivial if and only if $S$ contains one of the four configurations of $(-2)$-curves described in Theorem \ref{thm: Enriquesconfigs} and then $\Aut_{\nt}(S)$ is as in Theorem \ref{thm: Enriquesgroups}.
\end{theorem}

\begin{convention}
Following the mathematical taste of both authors we work over an algebraically closed field of arbitrary characteristic.
\end{convention}

\noindent \textbf{Acknowledgements:}
The second author gratefully acknowledges funding from the DFG under research grant MA 8510/1-1 and he would like to thank the Department of Mathematics at the University of Utah, where parts of this article were written, for its hospitality.
We thank Shigeyuki Kondo for pointing out some missing cases in Lemma \ref{lem: trivialonfiberautomorphisms} in an earlier version of this article.
 \section{Generalities}

 \subsection{Genus one fibrations}\label{subsection2.1} The proofs of all  the facts about genus one fibrations on algebraic surfaces for which we do not supply a proof  can be found in  \cite[Chapter 4]{CDL}. One also finds there  references to the original results. 
 
 A  complete geometrically integral curve over a field $K$ of arithmetic genus one is called elliptic (resp. quasi-elliptic) if it is smooth (resp. regular but non-smooth). We do not assume that it admits a rational point. An elliptic curve with a fixed rational point is an abelian curve, i.e. has a structure of an abelian variety of dimension one. A quasi-elliptic curve  contains a unique non-smooth point defined  over a radical extension of $\Bbbk(\eta)$ of degree $p = 2$ or $3$. It is isomorphic to a plane cuspidal curve over the algebraic closure of $K$.  If it admits  a rational point, then the complement  of this  point  admits a structure of a connected unipotent algebraic group of dimension $1$, a non-trivial inseparable form of the additive group $\bbG_{a,K}$. We refer to an elliptic or a quasi-elliptic curve $E$ as a \emph{genus one curve}. For any Dedekind scheme $B$ over an algebraically closed field $\Bbbk$ with field of fraction $K$, there exists a projective flat morphism $f:X\to B$ with generic fiber $X_\eta$  (i.e. the fiber over the generic point $\eta\in B$ with residue field $K$) isomorphic to $E$. Using resolution of singularities of two-dimensional excellent schemes, one may assume that $X$ is regular. The theory of birational transformations of such schemes due to Lichtenbaum and Shafarevich allows one to assume that $X$ is relatively minimal over $B$, i.e. any birational morphism $X\to X'$ over $X$ is an isomorphism. 
 
Let $X^\sharp$ denote the open subset of points on $X$ where $f$ is smooth. If $X_\eta$ is smooth, then $X\setminus X^\sharp$ consists of singular points of fibers over closed points. If $X_\eta$ is quasi-elliptic, then $X\setminus X^\sharp$ also contains the closure $\mathfrak{C}$ of the cusp of $X_\eta$. If $X_\eta(K)\ne \emptyset$, the closure  of a rational point on $X_\eta$ defines a section $f: B \to X^\sharp$. It equips $X^\sharp$ with a structure of a commutative group scheme over $B$. This group scheme coincides with the \emph{N\'eron model} of $(X^\sharp)_\eta$. If $X(K)=\emptyset$, then the jacobian variety $J(X_\eta)$ (the Rosenlicht generalized jacobian variety if $X_\eta$ quasi-elliptic) admits a relatively minimal model $J(f):J(X)\to B$ with generic fiber $J(X_\eta)$. It is called the \emph{jacobian fibration} of $f$. If one of the fibers is singular as it will be in our case, the group of rational points $X(K)$ is a finitely generated abelian group, called the \emph{Mordell-Weil group} of $f$ and denoted by $\MW(f)$.  Note that $\MW(f)$ is a $p$-group if $f$ is quasi-elliptic.

A choice of a group structure on $J_\eta^\sharp$ equips $X_\eta^\sharp$ with a unique structure of a \emph{torsor} (i.e. a principal homogeneous space) under $J_\eta^\sharp$. The \emph{index} of  a genus one fibration $f:X\to B$ is the greatest common divisor of the degrees of extensions $L/K$ such that $X_\eta^\sharp(L)\ne \emptyset$. The closure of such an $L$-valued point of $X_\eta^\sharp$ in $X$ is a (possibly inseparable) degree $m = [L:K]$ cover of the base $B$ and we call it an $m$-section. It is known that the index coincides with the \emph{period} of $X_\eta^\sharp$, i.e. the order of the isomorphism class of $X_\eta^\sharp$ in the Weil-Ch\^{a}telet group $H^1(K,J_\eta^\sharp)$ of isomorphism classes of torsors under $J_\eta^\sharp$.

A fibration that becomes isomorphic to the trivial fibration after a finite base change is called \emph{isotrivial}. If the fibration is elliptic, this is equivalent to the condition that the $j$-invariant of the jacobian fibration is constant.

The fibers of $f$ over closed points are curves of arithmetic genus one over $\Bbbk$. If $f$ is elliptic they are elliptic curves over $\Bbbk$, except for finitely many singular fibers. If $f$ is quasi-elliptic they are isomorphic to a cuspidal plane cubic over $\bbk$, except for finitely many which are reducible curves. 

We are assuming that the reader is familiar with Kodaira's classification of singular fibers of genus one  fibrations. We use the dictionary between Kodaira's  notation and the notation from \cite{CDL} based on the notations of extended Dynkin diagrams:
$$\begin{pmatrix}I_0 &II&III&IV&I_n(n> 1)&I_{n-4}^*(n\ge 4)&IV^*&III^*&II^*\\
\tilde{A}_0^{*}&\tilde{A}_0^{**}&\tilde{A}_1^{*}&\tilde{A}_2^{*}&\tilde{A}_{n-1}&\tilde{D}_{n}&\tilde{E}_{6}&\tilde{E}_{7}&\tilde{E}_{8}\end{pmatrix}
$$

The fibers of types $I_n$ with $n \geq 1$ are called of \emph{multiplicative type} and the other singular fibers are called of \emph{additive type}. Multiplicative fibers do not occur on quasi-elliptic fibrations. Any fiber $X_t$ can be considered as a Cartier or Weil divisor on $X$. As a Cartier divisor its local equation is given by the pull-back of a local parameter at $t$. As a Weil divisor, it is given by the linear combination of its irreducible components taken with multiplicity equal to the order of the vanishing of the local equation at its generic point. It is known that $X_t = m_t\bar{X}_t$, where one of the irreducible components of $X_t$ enters in $\bar{X}_t$ with multiplicity 1. The number $m_t$ is called the \emph{multiplicity} of the fiber and a fiber with $m_t > 1$ is called an $m_t$-multiple fiber. Otherwise it is called a simple or non-multiple fiber. To denote that a fiber of a given type is an $m_t$-multiple fiber, we write $m_t$ in front of its type notation.

If $X(K)\ne \emptyset$, all fibers are non-multiple, however the converse is not true in general. The isomorphism classes of  torsors with no multiple fibers form a subgroup of the Weil-Ch\^{a}telet group, called the \emph{Tate-Shafarevich group}. In the case when 
$H^1(J(X),\calO_{J(X)}) =  H^2(J(X),\calO_{J(X)}) = 0$ (e.g. if $X$ is a rational surface), this group is trivial, and the Weil-Ch\^{a}telet group is isomorphic to the direct sum of the Weil-Ch\^{a}telet groups $H^1(K_t,J(X)\otimes_KK_t)$, where $K_t/K$ is the extension of $K$ equal to the field of fractions of the completion of the local ring of $B$ at a closed point $t$. By Hensel's lemma, the class of the torsor in one of the latter groups is trivial if and only if the fiber $X_t$ is non-multiple.

If $X_t = m_t\bar{X}_t$ is a multiple fiber, the normal sheaf 
$\mathcal{N}_t = \calO_{X_t}(\bar{X}_t)$ defines an $m_t$-torsion element in the Picard group $\Pic(\bar{X}_t)$ of isomorphism classes of invertible sheaves on $\bar{X}_t$. If its order in this group is equal to $m_t$, the fiber is called \emph{tame}, otherwise it is a \emph{wild} fiber. A genus one fibration has  no wild fibers if and only if it is \emph{cohomologically flat} (in dimension 0), i.e. $h^0(\calO_{X_t}) = 1$ for all $t \in B$. If $m_t$ is prime to the characteristic $p$, then the fiber is tame. Also, it is tame if $H^1(X,\calO_X) = 0$. 

By the Ogg-Shafarevich theory for elliptic fibrations (see \cite[Chapter 4]{CDL}) the local invariants of tame $m_t$-multiple fibers in $H^1(K_t,X\otimes_KK_t)$ are in bijection with $m_t$-torsion line bundles on $\bar{X}_t$. We expect that the same is true for quasi-elliptic fibrations. Since this seems to be unknown at the moment, let us define the following.

\begin{definition} \label{def: quasiellipticoggshafarevich}
Let $f: X \to B$ be a genus one fibration. We say that $f$ satisfies condition ${\rm (OS)}$ if for every closed point $t \in B$, the subgroup of the local Weil-Ch\^{a}telet group $H^1(K_t,J(X)\otimes_KK_t)$ consisting of tame torsors of index $m_t$ is naturally isomorphic to the $m_t$-torsion subgroup of $J(X_t)^0$.
\end{definition}

Finally, if $X$ is a projective surface surface with no wild fibers, then its canonical sheaf $\omega_X$ is given by the formula
\begin{equation} \label{eq: canonicalbundleformula}
\omega_X = f^*(\calL\otimes \omega_B)\otimes \calO_X(\sum_{t\in B}(m_t-1)\bar{X}_t),
\end{equation}
where $\calL$ is an invertible sheaf of degree $-\chi(X,\calO_X)$.

 \subsection{Rational genus one fibrations}

 Now we specialize to our case, where we consider $X$ to be a smooth projective rational surface $S$ over $\Bbbk$. Obviously, $B = \bbP^1$ in this case. It follows from the discussion above that $f$ is cohomologically flat. First, let us describe some basic properties of multiple fibers on rational genus one surfaces.
 
 \begin{lemma} \label{multiplefiber}
Let $f: S \to \bbP^1$ be a rational genus one surface. Then, 
\begin{enumerate}
\item $f$ has at most one multiple fiber $mF_0$,
\item the canonical bundle satisfies $K_S = -F_0$, and
\item if $m > 1$, then the following hold:
\begin{enumerate}
\item If $F_0$ is smooth and supersingular, then $p \nmid m$.
\item If $F_0$ is multiplicative, then $p \nmid m$.
\item If $F_0$ is additive, then $m = p$.
\end{enumerate}
\end{enumerate}
\end{lemma}

\begin{proof}
As $f$ is cohomologically flat, the first two claims follow immediately from Equation \eqref{eq: canonicalbundleformula}.
Since the multiple fiber is not wild, its normal bundle $\calO_F(F_0)$ has order $m$. If $F_0$ is either supersingular or multiplicative, the $p$-torsion subgroup of $\Pic(F_0)$ is trivial, which means that $p \nmid m$. If $F$ is additive, then $\Pic(F)$ is $p$-elementary and thus $m = p$.
\end{proof}

The Noether formula $12\chi(S,\calO_S)= c_2+K_S^2$ implies that the second Chern number of $S$ is equal to 12. It coincides with the topological Euler-Poincar\'e characteristic of $S$ computed in \'etale topology if $\Bbbk\ne \bbC$. It is known that
$$c_2(S) = 12 = \sum_{t\in B}(e_t+\delta_t).$$
If $b_t$ denotes the number of irreducible components of $S_t$, then $e_t = 0$ if $S_t$ is smooth, $e_t = b_t$ if $S_t$ is multiplicative, $e_t = b_t + 1$ if $S_t$ is additive, and $e_t = b_t - 1$ if $f$ is quasi-elliptic. The number $\delta_t \ge 0$ can be nonzero only if $p= 2,3$, $f$ is elliptic, and $S_t$ is of additive type. 

Recall that a genus one fibration on a rational surface is called \emph{extremal} if one of the following equivalent conditions is satisfied:
\begin{itemize}
\item The Mordell-Weil group of the jacobian fibration is finite;
\item $\sum_{t\in \bbP^1}(b_t-1) = 8$.
\end{itemize}
Since the Mordell-Weil group of a quasi-elliptic surface is finite, $\sum_{t\in \bbP^1}e_t = 8$ holds for all quasi-elliptic surfaces. For extremal genus one fibrations, the structure of the Mordell-Weil group is determined by the structure of the groups of components of  reducible fibers: $\MW(J(f))^{\oplus 2}$ is isomorphic to the direct sum of these groups.

 The classification of all possible singular fibers of an elliptic fibration is known in all characteristics (see \cite{Persson} if $p\ne 2,3$ and \cite{Lang2}, \cite{Lang3} if $p = 2,3$). Also known is the classification of singular fibers of quasi-elliptic fibrations (see \cite{Ito1}, \cite{Ito2}, \cite[Chapter 4]{CDL}). 
For the convenience of the reader, we list possible singular fibers on rational elliptic or quasi-elliptic surfaces in the cases that will occur in our work in Table \ref{table:singularfibers}.

\begin{lemma} \label{lem: somefibers}
Let $f:S \to \bbP^1$ be a rational genus one surface. If $f$ admits a fiber $F$ such that
\begin{itemize}
\item $F$ is of multiplicative type with at least $8$ components, or
\item $F$ is of additive type with at least $7$ components,
\end{itemize} 
then the singular fibers of $f$ (resp. the non-cuspidal fibers if $f$ is quasi-elliptic) are as in one of the cases in the following Table \ref{table:singularfibers}.

\end{lemma}
\begin{table}[h]
\resizebox{\textwidth}{!}{
 $
 \begin{array}{|c|c|c|} 
 \hline
 F & \text{possible combinations of other singular (resp. non-cuspidal) fibers} & \text{Characteristic} \\ \hline
 \tilde{E}_8 &
 \begin{array}{c}
 (\tilde{A}_0^{**}), (\tilde{A}_0^*, \tilde{A}_0^*) \\
 \text{none}, (\tilde{A}_0^*)
 \end{array}
 &
  \begin{array}{c}
 \neq 2,3 \\
 2,3
 \end{array}
  \\ \hline 
 \tilde{D}_8 &
 \begin{array}{c}
 (\tilde{A}_0^*, \tilde{A}_0^*) \\\
\text{none}
 \end{array}
 &
  \begin{array}{c}
 \neq 2\\
 2
 \end{array}
  \\ \hline 
 \tilde{A}_8 &
 \begin{array}{c}
 (\tilde{A}_0^*,\tilde{A}_0^*, \tilde{A}_0^*) \\\
(\tilde{A}_0^{**})
 \end{array}
 &
  \begin{array}{c}
 \neq 3\\
 3
 \end{array}
  \\ \hline 
 \tilde{E}_7  &
 \begin{array}{c}
 (\tilde{A}_0^{*},\tilde{A}_0^{**}), (\tilde{A}_0^*,\tilde{A}_0^*, \tilde{A}_0^*) , (\tilde{A}_1^*),(\tilde{A}_1,\tilde{A}_0^*) \\
 (\tilde{A}_0^{**}), (\tilde{A}_0^*,\tilde{A}_0^*, \tilde{A}_0^*),(\tilde{A}_1^{*}), (\tilde{A}_1,\tilde{A}_0^{*})  \\
  \text{none}, (\tilde{A}_0^*, \tilde{A}_0^*),(\tilde{A}_1),(\tilde{A}_1^*)
 \end{array}
 &
 \begin{array}{c}
 \neq 2,3 \\
 3 \\
 2
 \end{array}
 \\ \hline 
 \tilde{D}_7 &
 \begin{array}{c}
(\tilde{A}_0^{**},\tilde{A}_0^*),(\tilde{A}_0^*,\tilde{A}_0^*, \tilde{A}_0^*) \\
(\tilde{A}_0^{**}),(\tilde{A}_0^*,\tilde{A}_0^*, \tilde{A}_0^*) \\
 \text{none}, (\tilde{A}_0^{*})
 \end{array}
 &
 \begin{array}{c}
 \neq 2,3 \\
 3 \\
 2
 \end{array}
 \\ \hline 
 \tilde{A}_7 &
 \begin{array}{c}
(\tilde{A}_0^{**},\tilde{A}_0^*, \tilde{A}_0^*),(\tilde{A}_0^*,\tilde{A}_0^*,\tilde{A}_0^*, \tilde{A}_0^*),(\tilde{A}_1,\tilde{A}_0^*, \tilde{A}_0^*) \\
(\tilde{A}_0^*,\tilde{A}_0^*,\tilde{A}_0^*,\tilde{A}_0^*),(\tilde{A}_1,\tilde{A}_0^*, \tilde{A}_0^*) \\
(\tilde{A}_0^{**}),(\tilde{A}_0^*,\tilde{A}_0^*,\tilde{A}_0^*, \tilde{A}_0^*),(\tilde{A}_1^*)
\end{array} &
\begin{array}{c}
 \neq 2,3 \\
 3 \\
 2
 \end{array}
\\ \hline
 \tilde{E}_6 &
  \begin{array}{c}
 (\tilde{A}_0^{**}, \tilde{A}_0^{**}),(\tilde{A}_0^{**},\tilde{A}_0^*, \tilde{A}_0^{*}), (\tilde{A}_0^*,\tilde{A}_0^*,\tilde{A}_0^*,\tilde{A}_0^{*}), (\tilde{A}_1,\tilde{A}_0^{**}),  (\tilde{A}_1, \tilde{A}_0^*,\tilde{A}_0^*) , (\tilde{A}_1^*,\tilde{A}_0^{*}), (\tilde{A}_2^*), (\tilde{A}_2,\tilde{A}_0^*)\\
 
\text{none}, (\tilde{A}_0^{**}), (\tilde{A}_0^*, \tilde{A}_0^*),(\tilde{A}_0^*,\tilde{A}_0^*, \tilde{A}_0^*),(\tilde{A}_1),(\tilde{A}_1^*),(\tilde{A}_1, \tilde{A}_0^*)(\tilde{A}_2),(\tilde{A}_2^*) \\

(\tilde{A}_0^{**}),(\tilde{A}_0^*,\tilde{A}_0^*,\tilde{A}_0^*,\tilde{A}_0^*),(\tilde{A}_1^*),(\tilde{A}_1, \tilde{A}_0^*, \tilde{A}_0^*), (\tilde{A}_2^*),(\tilde{A}_2,\tilde{A}_0^*) 
\end{array} &
\begin{array}{c}
 \neq 2,3 \\
 3 \\
 2
 \end{array}
 
 \\ \hline 
 \tilde{D}_6 &
  \begin{array}{c}
  (\tilde{A}_0^{**},\tilde{A}_0^{**}), (\tilde{A}_0^{**},\tilde{A}_0^*, \tilde{A}_0^{*}),  (\tilde{A}_0^*,\tilde{A}_0^*,\tilde{A}_0^*, \tilde{A}_0^{*}), (\tilde{A}_1^*,\tilde{A}_0^*), (\tilde{A}_1, \tilde{A}_0^*, \tilde{A}_0^*), (\tilde{A}_1, \tilde{A}_1) \\
  (\tilde{A}_0^{**}),(\tilde{A}_0^{**},\tilde{A}_0^*),(\tilde{A}_0^*,\tilde{A}_0^*,\tilde{A}_0^*,\tilde{A}_0^*),(\tilde{A}_1, \tilde{A}_0^*, \tilde{A}_0^*),(\tilde{A}_1, \tilde{A}_1) \\
  \text{none}, (\tilde{A}_0^*,\tilde{A}_0^*), (\tilde{A}_1), (\tilde{A}_1^*, \tilde{A}_1^*)
\end{array} &
\begin{array}{c}
  \neq 2,3 \\
 3 \\
 2
 \end{array}
\\  \hline 
\end{array}
$
}
\smallskip
\caption{Some configurations of singular fibers on rational genus one surfaces}
\label{table:singularfibers}
\end{table}

\subsection{Halphen pencils} \label{Halpencil}
A rational surface $S$ with a genus one fibration $f: S \to \bbP^1$ is a basic rational surface, i.e. it admits a birational morphism $\pi: S\to \bbP^2$. The morphism $\pi$ admits a factorization
\beq
\label{factorization}
\pi:S = S_{10}\overset{\pi_9}{\longrightarrow} S_{9}\overset{\pi_{8}}{\longrightarrow}\cdots \overset{\pi_2}{\longrightarrow} S_2\overset{\pi_1}{\longrightarrow} S_1 =\bbP^2,
\eeq
where each morphism $\pi_i:S_{i+1}\to S_{i}$ is the blow-up of one point $x_i\in S_{i}$. For any $j > i$, let $\pi_{j,i}:= \pi_i\circ\cdots\circ \pi_{j-1}:S_{j}\to S_i$. A point $x_j\in S_{j}$ with $\pi_{j,i}(x_j) = x_i\in S_i$ is called \emph{infinitely near} to $x_i$ (of order $j-i$). A curve of negative self-intersection on $S$ is either a smooth rational curve of self-intersection $-1$ (a \emph{$(-1)$-curve}), and in this case it is an $m$-section of $f$, or a smooth rational curve of self-intersection $-2$ (a \emph{$(-2)$-curve}) in which case it is an irreducible component of some fiber. This implies that at most 3 of the points $x_i$ are collinear, at most $6$ of them lie on a conic and there are no cubics passing through all points with a singular point at one of them.

The image of the  pencil of genus one curves defined by $f$ under $\pi$ is a pencil $\calP$ of curves of degree $3m$ with general member having singular points of multiplicity $m$ at the points $x_1,\ldots,x_9$ (appropriately interpreted if some of the $x_i$ are infinitely near to each other). This pencil of curves in $\mathbb{P}^2$ is called a \emph{Halphen pencil} of index $m$. If $m > 1$, the image of the multiple fiber is the unique cubic curve passing through the points $x_1,\ldots,x_9$ taken with multiplicity $m$. The points $\pi_{j,1}(x_j), j = 1,\ldots,9,$ are the intersection points of two general members of the Halphen pencil defined by the elliptic fibration on $S$. Their number can be less than $9$. 
The points $\pi_{j,2}(x_i), j = 2,\ldots,9,$ are the intersection points of the proper transforms of two general members of the pencil in $S_2$, and so on, until the proper transforms of the general members of the pencil on $S$ do not intersect.

The set $\Bs(\calP) = \{x_1,\ldots,x_9\}$ is the set of base points of the Halphen pencil $\calP$.  We  order the set $\Bs(\calP)$  as follows:
$$(x_1,\ldots,x_9) = (x_1^{(1)},\ldots,x_1^{(k_1)},x_2^{(1)},\ldots,x_2^{(k_2)},\ldots,x_s^{(1)},\ldots,x_s^{(k_s)}),$$
where
$$x_1^{(k_1)}\succ x_1^{(k_1-1)}\succ\cdots\succ x_1^{(1)} = q_1,\quad x_2^{(k_2)}\succ \cdots\succ x_2^{(2)}\succ x_2^{(1)} = q_2,\quad  x_s^{(k_s)}\succ \cdots\succ x_s^{(2)}\succ x_s^{(1)} = q_s.$$
Here $x_j^{(i)}\succ x_j^{(i-1)}$ is infinitely near to $x_j^{(i-1)}$ of order $1$. We have
$$k_1+\cdots+k_s = 9.$$

Let $\calE_i:=\pi_{10,i}^{-1}(x_i)$. It is called the \emph{exceptional configuration} over the base point $x_i$. We have $\calE_{i}\subseteq \calE_{i-1}$ if $x_i\succ x_{i-1}$ and the difference is a $(-2)$-curve. In our ordering we have $s$  chains 
$$\calR_i = R_i^{(1)}+\cdots+R_i^{(k_i-1)}$$
of such curves of lengths $k_i-1, i = 1,\ldots,s$ and $s$ irreducible exceptional configurations $E_i$ (which are $(-1)$-curves) that intersect $R_i^{(k_i-1)}$ with multiplicity 1. All other  exceptional configurations are of the form 
$E_i+R_i^{(k_i-1)}+\cdots+R_i^{(j)}, i = 1,\ldots,s, j = 1,\ldots,k_i-1$.

$$
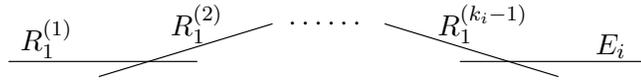
\xy (-10,0)*{};
(0,0)*{};(25,0)*{}**\dir{-};(60,0)*{};(85,0)*{}**\dir{-};
(12,-2)*{};(35,5)*{}**\dir{-};(50,5)*{};(73,-2)*{}**\dir{-};
(5,3)*{R_1^{(1)}};(25,5)*{R_1^{(2)}};(63,5)*{R_1^{(k_i-1)}};(80,2)*{E_i};(42,5)*{\cdots\cdots};
\endxy
$$

Of course, the birational morphism $\pi:S\to \bbP^2$ is not unique, so the shape of the corresponding Halphen pencil and its base points could be different. Two Halphen pencils 
corresponding to the same surface $S$ are obtained from each other by a Cremona transformation of the plane.


 A choice of $\pi: S \to \bbP^2$ defines a \emph{geometric basis} $(e_0,e_1,\ldots,e_9)$ in $\Pic(S)$. Here $e_0 = c_1(\pi^*\calO_{\bbP^2}(1))$ and the other $e_i$'s are the divisor classes of the exceptional configurations $\calE_i$. We have 
$$e_0^2 = 1, \quad e_i^2 = -1,\  i\ne 0, \quad e_i\cdot e_j = 0,\  i\ne j$$  
and
$$K_S = -3e_0+e_1+\cdots+e_9.$$
As we explained in the introduction, a geometric basis defines a marking $\Pic(S)\to \sfI^{1,9}$ with the image of $K_S$ equal to the generator $\frakf$ of the radical of the orthogonal complement $\frakf^\perp$ in $\sfI^{1,9}$.

 \subsection{Some relevant fixed loci} 
The following Lemma is an elementary fact from linear algebra:

\begin{lemma}\label{fixedpoints} Let $T$ be an automorphism of $\bbP^2$ of finite order $n$ and let $\mathcal{F}$ be its set of fixed points. Then one of the following cases occurs:
\begin{itemize}
\item[(a)] $(p,n) = 1$ and $\mathcal{F}$ is the union of a line $L$ and a point outside $L$;
\item[(b)] $(p,n) = 1$, $n \geq 3$ and $\mathcal{F}$ consists of 3 non-collinear isolated points $q_1,q,q'$;
\item[(c)] $n = pk, (k,p) = 1$ and the fixed locus of $T^k$  is a point;
\item[(d)] $n = pk, (k,p) = 1$ and the fixed locus of $T^k$  is  a line $L$.
\end{itemize}
Every $T$-invariant line meets $\mathcal{F}$ in at least $1$ (if $(p,n) \neq 1$) or $2$ (if $(p,n) = 1$) points.
\end{lemma}
%
%
%
%
%

Recall the structure of the fixed loci of irreducible curves of arithmetic genus one.

 \begin{lemma}\label{autoelliptic}  Let $E$ be an irreducible curve of arithmetic genus one over an algebraically closed field. For $g \in G$, let $E_{red}^g$ be the reduced fixed locus of $g$ and let $E^\#$ be the smooth locus of $E$.
 \begin{enumerate}
\item If $p\ne 2,3$ and $E$ is elliptic
\begin{center}
\begin{tabular}{|>{\centering\arraybackslash}m{1.8cm}|>{\centering\arraybackslash}m{3cm}|>{\centering\arraybackslash}m{2.5cm}|>{\centering\arraybackslash}m{2.5cm}|}
\hline
$j$ & $G$ & $\ord(g)$ & \vspace{1mm} $E_{red}^g$\\ [1mm] \hline
$\neq 0,1$ & $\bbZ/2\bbZ$ & $2$ & $(\bbZ/2\bbZ)^2$ \\ \hline
$1$ & $\bbZ/4\bbZ$ & 
$\begin{cases}
2 \\
4
\end{cases}$ &
$\begin{cases}
(\bbZ/2\bbZ)^2 \\
\bbZ/2\bbZ
\end{cases}$ \\ \hline
$0$ & $\bbZ/6\bbZ$ & 
$\begin{cases}
2\\
3 \\
6
\end{cases}$ &
$\begin{cases}
(\bbZ/2\bbZ)^2 \\
\bbZ/3\bbZ \\
\{*\}
\end{cases}$ \\ \hline
\end{tabular}
\end{center}
  \item If $p = 3$ and $E$ is elliptic
\begin{center}
\begin{tabular}{|>{\centering\arraybackslash}m{1.8cm}|>{\centering\arraybackslash}m{3cm}|>{\centering\arraybackslash}m{2.5cm}|>{\centering\arraybackslash}m{2.5cm}|}
\hline
$j$ & $G$ & $\ord(g)$ & \vspace{1mm} $E_{red}^g$\\ [1mm] \hline
$\neq 0$ & $\bbZ/2\bbZ$ & $2$ & $(\bbZ/2\bbZ)^2$ \\ \hline
$0$ & $\bbZ/3\bbZ \rtimes \bbZ/4\bbZ$ & 
$\begin{cases}
2 \\
3,6 \\
4
\end{cases}$ &
$\begin{cases}
(\bbZ/2\bbZ)^2 \\
\{*\} \\
\bbZ/2\bbZ
\end{cases}$ \\ \hline
\end{tabular}
\end{center}
 \bigskip
  \item If $p = 2$ and $E$ is elliptic
\begin{center}
\begin{tabular}{|>{\centering\arraybackslash}m{1.8cm}|>{\centering\arraybackslash}m{3cm}|>{\centering\arraybackslash}m{2.5cm}|>{\centering\arraybackslash}m{2.5cm}|}
\hline
$j$ & $G$ & $\ord(g)$ & \vspace{1mm} $E_{red}^g$\\ [1mm] \hline
$\neq 0$ & $\bbZ/2\bbZ$ & $2$ & $\bbZ/2\bbZ$ \\ \hline
$0$ & $Q_8 \rtimes \bbZ/3\bbZ$ & 
$\begin{cases}
2,4,6 \\
3
\end{cases}$ &
$\begin{cases}
\{*\} \\
\bbZ/3\bbZ
\end{cases}$ \\ \hline
\end{tabular}
\end{center}
\bigskip
 \item If $E$ is cuspidal
\begin{center}
\begin{tabular}{|>{\centering\arraybackslash}m{1.8cm}|>{\centering\arraybackslash}m{3cm}|>{\centering\arraybackslash}m{2.5cm}|>{\centering\arraybackslash}m{2.5cm}|}
\hline
$G$ & $g$ & \vspace{1mm} $|E_{red}^g|$ & $E_{red}^g \cap E^\#$\\ [1mm] \hline
$\bbG_a \rtimes \bbG_m$ & $\begin{cases}
\in \bbG_a \\
\not \in \bbG_a
\end{cases}$ & $\begin{cases}
1 \\
2
\end{cases}$ & $\begin{cases}
\emptyset \\
\{*\}
\end{cases}$ \\ \hline
\end{tabular}
\end{center}
\bigskip
 \item If $E$ is nodal
\begin{center}
\begin{tabular}{|>{\centering\arraybackslash}m{1.8cm}|>{\centering\arraybackslash}m{3cm}|>{\centering\arraybackslash}m{2.5cm}|>{\centering\arraybackslash}m{2.5cm}|}
\hline
$G$ & $g$ & \vspace{1mm} $|E_{red}^g|$ & $E_{red}^g \cap E^\#$\\ [1mm] \hline
$\bbG_m \rtimes \bbZ/2\bbZ$ & $\begin{cases}
\in \bbG_m \\
\not \in \bbG_m, p \neq 2 \\
\not \in \bbG_m, p = 2
\end{cases}$ & $\begin{cases}
1 \\
3 \\
2
\end{cases}$ & $\begin{cases}
\emptyset \\
\bbZ/2\bbZ \\
\{*\}
\end{cases}$ \\ \hline
\end{tabular}
\end{center}

\end{enumerate}
\end{lemma}

\section{Action of $\Aut(S)$ on the jacobian surface}\label{action}
\label{sec: actiononjacobian}
All the facts mentioned in this subsection  without a reference to their proofs can be found in \cite{Raynaud}. 

Let $f:X\to B$ be a genus one fibration as in Subsection \ref{subsection2.1}. For any $B$-scheme $T\to B$ we denote by $f_T: X_T \to T$ the morphism obtained from $f$ by base change along $T\to B$. We denote by $\calP_{X/B}$ the relative \emph{Picard functor}. It is a sheaf in the  fppf-topology on $B$ associated to the pre-sheaf $T\to \Pic(X_T)$. It is automatically a sheaf in the \'etale topology and coincides  with the sheaf $R^1f_*\bbG_m$. There is a natural injective homomorphism $\Pic(X_T)/f_T^*(\Pic(T))\to \calP_{X/B}(T)$, which is bijective if $X(B)\ne \emptyset$ (or if the cohomological Brauer group $H^2(T,\bbG_m)$ of $T$ is trivial).

For every  point $t \in B$ including the generic point, the fiber of $\calP_{X/B,t}$ is representable by a one-dimensional commutative group scheme locally of finite type $\mathbf{\Pic}_{X_t/t}$, the \emph{Picard scheme} of $X_t$. Its connected component of the identity $ \mathbf{\Pic}_{X_t/t}^0$ is the (generalized) jacobian of the curve $X_t$ and denoted by $J(X_t)$. If $X_t = \sum m_t^{(i)}X_t^{(i)}$, then $J(X_t)(\Bbbk)$ is the group of isomorphism classes of invertible sheaves $\calL$ on $X_t$ such that $\deg \calL\otimes_{\calO_X}\calO_{X_t^{(i)}} = 0$  for all $i$. Recall that the degree of an invertible sheaf $\calL$ on a complete curve $Z$ over a field is defined to be $\chi(Z,\calL)-\chi(Z,\calO_Z)$.

\begin{lemma}\label{oort} Suppose $f:X\to B$ is cohomologically flat. Let $X_t = m_t\bar{X}_t$. Then the restriction homomorphism 
$$\mathbf{\Pic}_{X_t/t} \to \mathbf{\Pic}_{\bar{X}_t/t}$$
is an isomorphism.
\end{lemma}

\begin{proof}  By one of the properties of cohomological flatness, we have $h^0(X_t,\calO_{X_t}) = 1$. Since the Euler-Poincar\'e characteristic of the structure sheaves of geometric fibers is constant, we obtain $h^1(X_T,\calO_{X_t}) = 1$. It is known that the linear space $H^1(X_t,\calO_{X_t})$ has a structure of a Lie algebra isomorphic to  the Lie algebra of the Picard scheme of $X_t$. Thus $\dim \mathbf{\Pic}_{X_t/t} = 1$. The kernel of the restriction homomorphism is a unipotent algebraic group equal to an extension of vector spaces of the cohomology groups of some nilpotent sheaves on $X_t$ (see \cite[Section 6, Corollaire]{Oort}). Then, $\dim \mathbf{\Pic}_{X_t/t} = 1$ shows that the dimension of these spaces is zero, and hence the kernel is trivial. Since $\dim \mathbf{\Pic}_{\bar{X}_t/t} = 1$, this proves the claim.
\end{proof}

We denote by $\calP_{X/B}'$ (resp. $\calP_{X/B}^0$) the subsheaf of $\calP_{X/B}$ whose values on any $T \to B$ are elements of $\calP_{X/B}(T)$ such that their restriction  to any fiber of $X_T\to T$ are the isomorphism classes of invertible sheaves of degree 0 (resp. degree 0 on each irreducible component of the fiber).  In particular,  the fiber of $\calP_{X/B}^0$ over any point $t\in B$  is representable by the jacobian variety $J(X_t)$. 

It is a natural guess that the Picard functor $\calP_{X/B}'$ (resp. $\calP_{X/B}^0$) must be related to the group scheme $J^\sharp(X) \to B$ (resp. its connected component of identity). However, neither of these functors is representable by a separated group scheme over $B$, unless all fibers of $f$ are non-multiple. The reason is very simple: If $f$ admits a multiple fiber $mF$, then the restriction map $\calP_{X/B}(B) \to \calP_{X/B}(\Spec k(B))$ is not injective, since the ideal sheaf of $F$ is in its kernel.

To remedy this situation, Raynaud considers the subsheaf $\calE_{X/B}$ of $\calP_{X/B}$ whose sections over an open subset $U$ (in the \'etale topology) restrict to the identity section of $\calP_{X/B}$ over the generic point of $U$. Then he proves that the quotient sheaf 
$\calQ_{X/B} = \calP_{X/B}/\calE_{X/B}$ is the maximal separated quotient sheaf of $\calP_{X/B}$ and it is representable by a separated group scheme, locally of finite type over $B$. We have an exact sequence of sheaves of abelian groups
\begin{equation}\label{ses: EPQ}
0\to \calE_{X/B}\to \calP_{X/B}\to \calQ_{X/B}\to 0.
\end{equation}
Let $\calQ_{X/B}'$ (resp. $\calQ_{X/B}^0$) be the image of $\calP_{X/B}'$ (resp. $\calP_{X/B}^0$). Then Raynaud proves that $\calQ_{X/B}'$ is representable by the group scheme $J^\sharp(X)$ and we have $\calQ_{X/B}^0 = (J^\sharp(X))^0$.

%

For any closed point $t$ of $B$, let $B(t) = \Spec \calO_{B,t}^h$ be the spectrum of the henselization of the local ring $\calO_{B,t}$ (or its formal completion) and $f(t): X(t) = X\times_BB(t)\to B(t)$ be the base change morphism. Since the base change is formally \'etale, the scheme $X(t)$ is regular and $X(t)\to B(t)$ is a genus one fibration with local base $B(t)$. The fiber over the closed point of $B(t)$ is isomorphic to the fiber $X_t$ under the projection $X(t)\to X$. We denote by $\eta_t$ the generic point of $B(t)$ and by $K_t$ its residue field, the field of fractions of $\calO_{B,t}^h$. 

Let $D_t$ be the free abelian group generated by the irreducible components $X_t^{(i)}$ of the fiber $X_t$ over the closed point of $B(t)$. Let  $X_t^{(i)}\cdot X_t^{(j)}$  be equal to the degree of $\calO_{X_t^{(i)}}$ on $X_t^{(j)}$. One shows that the matrix $(X_t^{(i)}\cdot X_t^{(j)})$  defines a symmetric bilinear form on $D_t$. Let $i:D_t\to \Pic(X(t))$ be the natural map that sends a component $X_t^{(i)}$ to the isomorphism class of the invertible sheaf $\calO_{X(t)}(X_t^{(i)})$. Its kernel is generated by $[X_t] = \sum m_t[X_t^{(i)}]$ and its image $\bar{D}_t$ is a finitely generated subgroup of $\Pic(X(t))$ with torsion subgroup of order $m_t$ generated by $\calO_{X(t)}(\bar{X}_t)$. The symmetric bilinear form on $D_t$ defines a symmetric bilinear form on $\bar{D}_t/\Tors$ and equips it with a  structure of an even negative definite quadratic lattice, a root lattice of type $A_n,D_n,E_n$, where $n$ is equal to the subscript in the notation of the type of the fiber. 

Let $\Disc_t$ denote the discriminant group of this lattice. It is a finite abelian group isomorphic to the group of symmetries of the  set of irreducible components of $X_t$ preserving the incidence relation. For example, if $X_t$ is of type $\tilde{A}_n$, the discriminant group is isomorphic to $\bbZ/(n+1)\bbZ$ and if $X_t$ is of type $\tilde{E}_8$, then $\Disc_t$ is trivial. The order of $\Disc_t$ is equal to the number of reduced components of $\bar{X}_t$.

Now we are ready to describe the sheaf  $\calE_{X/B}$. 

\begin{proposition}\label{prop: SheafE} Assume that $B = B(t)$ is local and $f:X = X(t)\to B$ is cohomologically flat.  Then  
\begin{enumerate}
\item $(\calE_{X/B})_t \cong \bar{D}_t$,
\item $(\calE_{X/B})_t\cap (\calP_{X/B}^0)_t = \Tors(\bar{D}_t) \cong \bbZ/m_t\bbZ$,
\item $(\calQ_{X/B}')_t/(\calQ_{X/B}^0)_t \cong  \Disc_t$,
\item $\calQ_{X/B}'$ is the  N\'eron model of $J(X_\eta)$, i.e. it coincides with the sheaf $(i_\eta)_*(J(X_\eta))$, where $i_\eta:\eta\hookrightarrow B$ is the inclusion morphism of the generic point. In particular $\calQ_{X/B}'(B) = J(X_\eta)(K)$.
\end{enumerate}
\end{proposition}

In particular, there is a natural isomorphism
\beq\label{raynaud2}
\Phi_t:\Disc_t \to J(t)^\sharp/J(t)^0.
\eeq
We note that one can prove the existence of the isomorphism $\Phi_t$ without the assumption that $f$ is cohomologically flat \cite[Theorem 9.6.1]{Bosch}. This yields the fundamental fact that the types of fibers of $f:X\to B$ and $J(f):J(X)\to B$ over the same point are the same even for multiple fibers.

Suppose $g\in \Aut(X)$ preserves the genus one fibration (this is automatic if $B$ is local). Assume first that $g$ acts identically on the base $B$, i.e. $g\in \Aut(X/B)$. Then, for any $T\to B$, the automorphism $g$ defines an automorphism $g\times \id_T:X_T\to X_T$ and hence a homomorphism $g^*:\Pic(X_T)\to \Pic(X_T)$. Passing to   the associated sheaf we obtain an automorphism of sheaves $g^*:\calP_{X/B}\to \calP_{X/B}$. If $g$ does not act identically on $B$ we modify the definition of $g^*$ as follows. Let $h:B\to B$ be the image of $g$ in $\Aut(B)$. Then $h$ defines an isomorphism of sheaves $\calP_{X/B}\to h^*(\calP_{X/B})$. On the other hand, we have  a canonical isomorphism $h^*(\calP_{X/B}) \to \calP_{X'/B}$, where $X'\to B$ denotes the base change with respect to $h:B\to B$ (\cite[Example 9.4.4]{Kleiman}).  Now, for any $T\to B$, the map $X_T\to X'_T$ defines  a functorial isomorphism $g_T^*:\Pic(X'_T)\to \Pic(X_T)$, hence an isomorphism of sheaves $g':\calP_{X'/B}\to \calP_{X/B}$. Composing it with $h^*$, we get an isomorphism of sheaves $$g^*:\calP_{X/B}\to \calP_{X/B}.$$

Obviously,  $g^*$ leaves invariant the subsheaves $\calP_{X/B}'$ and $\calE_{X/B}$ and hence defines an isomorphism of group schemes $\varphi(g): J^\sharp(X) = \calQ_{X/B}' \to \calQ_{X/B}' = J^\sharp(X)$. If $f$ is quasi-elliptic, $\varphi(g)$ extends to an isomorphism at the generic point of the curve $\mathfrak{C}$ of cusps on the regular surface $J(X)$. Thus $\varphi(g)$ is an isomorphism on an open subset of $J(X)$ whose complement is supported in a finite set of singular fibers. Since $f$ is a relatively minimal model, it is easy to see that $\varphi(g)$ extends to an automorphism of $X$. 

Therefore, denoting by $\Aut_f(X)$ the group of automorphisms of $X$ preserving the fibration $f$, we obtain a homomorphism
\beq\label{varphi}
\varphi: \Aut_f(X) \to \Aut_{J(f)}(J(X)).
\eeq
The main tools for our analysis of automorphisms of non-jacobian elliptic surfaces are the following properties of the homomorphism $\varphi$. Recall that for a surface $X$, the notation $\Num(X)$ denotes the group of numerical equivalence classes of divisors on $X$.

\begin{theorem}\label{thm: actiononjacobian}
Let $f: X \to B$ be a genus one surface with jacobian fibration $J(f):J(X) \to B$ and let $\Aut_f(X)$ be the group of automorphisms of $X$ preserving $f$. Assume that $f$ is cohomologically flat. Then, there is a homomorphism $\varphi: \Aut_f(X) \to \Aut_{J(f)}(J(X))$ satisfying the following properties, where $g \in \Aut_f(X)$:
\begin{enumerate}
\item Both $g$ and $\varphi(g)$ induce the same automorphism of $B$.
\item $\Ker(\varphi) \cong \MW(J(f))$.
\item $\varphi(g)$ preserves the zero section of $J(f): J(X) \to \bbP^1$.
\item If $g$ acts trivially on $\Num(X)$, then $\varphi(g)$ acts trivially on $\Num(J(X))$.
\item Let $mF_0$ be a fiber of $f$ of multiplicity $m$ and let $(J^{\sharp}_0)^0$ be the identity component of the smooth part $J_0^\sharp$ of the corresponding fiber $J_0$ of $J(f)$. If $g$ preserves $F_0$, then either $\varphi(g)$ acts trivially on $(J^{\sharp}_0)^0$ or one of the following holds, where $n = \ord(\varphi(g)|_{(J^{\sharp}_0)^0})$:
\begin{enumerate}
\item $F_0$ is smooth, $m = n =3$, $p \neq 3$.
\item $F_0$ is smooth, $m = 2, n \in \{2,4\}, p \neq 2$.
\item $F_0$ is smooth and ordinary, $m = n = p = 2$.
\item $F_0$ is an irreducible nodal curve, $n = m = 2$, $p \neq 2$.
\item $F_0$ is of type $\tilde{A}_1$, $n = m = 2$, $p \neq 2$.
\end{enumerate}  
\end{enumerate}
\end{theorem}

\begin{proof}
We define $\varphi$ as explained right before the theorem. Then, Property (1) follows from the construction of $\varphi$.
 
To see (2), we use that $\Ker(\varphi)$ acts trivially on $B$ by Property (1). Moreover, $g\in \Ker(\varphi)$ is an automorphism of $X_\eta$ over $\eta$ that acts  identically on its jacobian $J_\eta^\sharp$. It is well-known that this action is trivial if and only if $g$ comes from a translation, that is, if and only if $g \in {\rm MW}(J(f))$.

To see Property (3) we use that the action of $\varphi(g)$ on $J^\sharp(X)$ is via the action of $g$ on $\calP_{X/B}'$ that obviously preserves the zero section. Now $g$ preserves the subsheaf $\calE_{X/B}$ and the action of $g$ on $\calP'_{X/B}(B)$ defines the action of $\varphi(G)$ on 
$\calQ'_{X/B}(B) = J^\sharp(X)(B) = J_\eta(K) = \MW(J(f))$.   

Let us prove Property (4). We have $\MW(J(f)) = J_\eta(K) = \calQ'_{X/B}(B)$. Since the $H^i_{\et}(B,\calE_{X/B})$ are torsion groups, the short exact sequence \eqref{ses: EPQ} induces an isomorphism $\calP'_{X/B}(B) \otimes \bbQ \cong \calQ'_{X/B}(B) \otimes \bbQ$. Next, we use that the Brauer group of $B$ is trivial by Tsen's theorem, so that we have $\calP'_{X/B}(B) \subseteq \calP_{X/B}(B) = \Pic(X)/f^*\Pic(B)$ and hence $\MW(J(f)) \otimes \bbQ$ is a subquotient of $\Num(X) \otimes \bbQ$ in a natural way. 
In particular, if $g$ acts trivially on $\Num(X)$, then $\varphi(g)$ acts trivially on $\MW(J(f)) \otimes \bbQ$. We also know that $g$ preserves all reducible fibers $X_t$ and hence $\varphi(g)$ acts on $J_t^\sharp$. Now, since $g$ acts trivially on the discriminant group $\Disc_t$, the bijectivity of the map \eqref{raynaud2} implies that $\varphi(g)$ acts trivially on the discriminant groups of reducible fibers of $J(f)$, and hence $\varphi(g)$ preserves all components of reducible fibers and acts trivially on $\MW(J(f)) \otimes \bbQ$. Since these divisor classes generate $\Num(J(X)) \otimes \bbQ$, we have shown that $\varphi(g)$ acts trivially on $\Num(J(X)) \otimes \bbQ$ and hence also on $\Num(J(X))$ itself.

To prove (5), we use that there is an isogeny 
 $\lambda: J(F_0) = J(mF_0) = (\calP_{X/B})_0^0 \to (\calQ'_{X/B})_0^0 = (J^{\sharp}_0)^0$ with kernel $\bbZ/m\bbZ$ by Proposition \ref{prop: SheafE} and Lemma \ref{oort}, and that, by our construction of $\varphi$, the automorphism $\varphi(g)|_{(J^{\sharp}_0)^0}$ comes from an automorphism $h$ of $J(mF_0)$. Moreover, since $g$ preserves $F_0$, the automorphism $h$ fixes the $m$-torsion point on $J(F_0) = J(mF_0)$ corresponding to the sheaf $\calO_{mF_0}(F_0)$, which generates the kernel of $\lambda$. Note that $\varphi(g)|_{(J^{\sharp}_0)^0}$ and $h$ have the same order, because $h$ is no translation. Since $h$ fixes an $m$-torsion point, we have $m \leq 3$ by Lemma \ref{autoelliptic}. More precisely, by the same Lemma, we get the following cases in which $\varphi(g)$ can be non-trivial:

If $m = 3$, then $\ord(h) = 3$, $p \neq 3$, and $J(mF_0)$ is an elliptic curve with absolute invariant $0$. If $m = 2$ and $F_0$ is an elliptic curve, then either $\ord(h) = 4$, $p \neq 2$, and $J(mF_0)$ is an elliptic curve with absolute invariant $1$, or $\ord(h) = 2$ and $J(mF_0)$ is an elliptic curve which satisfies $j \neq 0$ if $p = 2$.
Finally, if $m = 2$ and $F_0$ is multiplicative, then $\ord(h) = 2$, $p \neq 2$ and $h$ interchanges the two boundary points in the regular compactification of $(J^{\sharp}_0)^0$. But since $\varphi(g)$ preserves the irreducible components of $J^{\sharp}_0$, the automorphism $h$ can interchange the two boundary points only if $J^{\sharp}_0$ has at most two components, that is, only if $F_0$ is of type $\tilde{A}_0^*$ or $\tilde{A}_1$.
\end{proof}

\section{The subgroup $\Aut_E(S)$} \label{sec: aute}
From now on, unless specified otherwise, we will consider a rational genus one surfaces $f:S \to \bbP^1$ of index $m \geq 1$ and we let $mF_0$ be its multiple fiber (or any fiber if $m = 1$).

Fix a $(-1)$-curve $E$ on $S$ considered as an $m$-section and let $\Aut_E(S)$ be the stabilizer subgroup of $E$ in $\Aut(S)$. If $m = 1$, then for any $g\in \Aut(S)$ there exists a unique element $E'$ in the Mordell-Weil group such $g(E) = t_{E'}(E)$, where $t_{E'}$ is the translation automorphism of $S$ defined by the section $E'$. This shows that
\beq\label{semidirect}
\Aut(S) = \MW(f)\rtimes \Aut_E(S).
\eeq
Moreover, since $\MW(f)$ acts transitively on the set of sections, we obtain that all subgroups of the form $\Aut_E(S)$ for different choices of $E$ are conjugate by an element from $\MW(f)$. We refer the reader to \cite{Karayayla1}
and \cite{Karayayla2} for a classification of automorphism groups of jacobian rational elliptic surfaces over the complex numbers based on the decomposition  \eqref{semidirect}.

However, if $m > 1$, \eqref{semidirect} may not be true anymore because $\Aut(S)$ may not act transitively on the set of $(-1)$-curves on $S$. Nevertheless, by Proposition \ref{thm: actiononjacobian}, we have an exact sequence
$$
0 \to \MW(J(f)) \cap \Aut_E(S) \to \Aut_E(S) \overset{\varphi}{\to} \Aut_0(J(S)),
$$
where $J(f): J(S) \to \bbP^1$ is the jacobian fibration and $\Aut_0(J)$ is the group of automorphisms fixing its zero section. It is not clear from Proposition \ref{thm: actiononjacobian} which conjugacy classes of subgroups of $\Aut_0(J(S))$ are realized as $\varphi(\Aut_E(S))$.

Choose a geometric marking on $S$ and let $\Aut(S)^*$ be the image of $\Aut(S)$ in $W(\sfE_9)$. A choice of a $(-1)$-curve $E$ defines a blow-down $S\to \calD$ to a weak del Pezzo surface $\calD$  of degree 1. This gives an embedding of $K_{\calD}^\perp = \sfE_8$ in $\sfE_9$ and gives a splitting of the surjection $W(\sfE_9)\to W(\sfE_8)$. This yields a well-known expression of $W(\sfE_9)$ as the semi-direct product 
\beq\label{semidirectweyl}
W(\sfE_9) = \sfE_8\rtimes W(\sfE_8).
\eeq
So, we see that there is a natural homomorphism $\Aut_E(S)\to W(\sfE_8)$ whose kernel is the subgroup $\Aut_{\ct}(S)$.  Since $W(\sfE_8)$ is finite, we have the following immediate corollary.

\begin{proposition} The group $\Aut_E(S)$ is finite if and only if $\Aut_{\ct}(S)$ is finite. The quotient group $\Aut_E(S)/\Aut_{\ct}(S)$ is isomorphic to a finite subgroup of $W(\sfE_8)$
\end{proposition}

As explained above, $\Aut_E(S)$ can be realized as a subgroup of the group of automorphisms of a weak del Pezzo surface $\calD$. In fact, we can say more about this surface $\calD$:

Let $G$ be a finite group. Recall that a $G$-surface is a pair $(X,G)$ consisting of a smooth projective surface $X$ with a faithful action of a finite group $G$. The $G$-surfaces form, naturally, a category with morphisms being $G$-equivariant morphisms of the underlying surfaces. A $G$-surface is called minimal if any birational morphism $(X,G)\to (X',G)$ is an isomorphism. If $X$ is a rational surface, then there is a birational morphism $(X,G)\to (X_0,G)$, where $X_0$ is either a del Pezzo surface with $\Pic(X_0)^G = \bbZ K_{X_0}$ or $X_0$ admits a structure of a conic bundle $X_0\to \bbP^1$ with $\Pic(X_0)^G \cong \bbZ^2$. This is all well-known and used, for example, in the classification of conjugacy classes of finite subgroups of the plane Cremona group (see \cite{DolgIsk}). Applying this to our case, we see that any finite subgroup $G$ of $\Aut_E(S)$ is a lift of a finite subgroup $G$ of a minimal del Pezzo $G$-surface $(\calD,G)$. Even more, we know that $G$ preserves $E$, and hence it fixes the image of $E$ in $\calD$. If $p =0$ or sufficiently large, all such groups were classified in \cite{DolgachevDuncan}.

 Elaborating on the ideas above, we obtain the following.

\begin{theorem}\label{delPezzo}
Let $f: S \to \bbP^1$ be a rational genus one surface of index $m \geq 1$ and let $E$ be a $(-1)$-curve on $S$. Then, the following hold:
\begin{enumerate}
 \item There is a weak del Pezzo surface $\calD$ and a birational morphism $\phi:S\to \calD$ such that $\Aut_E(S)$ is a lift of a subgroup $G$ of $ \Aut(\calD)$ that fixes a point on $\calD$. If $G$ is finite, then $(\calD,G)$ can be taken to be a minimal $G$-surface.
 \item If $m = 1$, then one can choose the weak del Pezzo surface $\calD$ in (1) such that $\Aut_E(S) = \Aut(\calD)$.
 \item If $m > 1$, let $mF$ be the multiple fiber of $f$. Then, there is a jacobian rational genus one surface $f': S' \to \bbP^1$ which admits a fiber $F'$ of the same type as $mF$
 such that $\Aut_E(S) \subseteq \Aut_0(S')$ is the stabilizer of a point $P$ of exact order $m$ on the identity component of $F'$. Moreover, $\Aut_{\ct}(S)$ is the stabilizer of $P$ for the action of $\Aut_{\ct}(S')$.
\end{enumerate}

\end{theorem}

\begin{proof}
Part (1) is immediate from the discussion before the theorem.

As for (2) and (3), let $\pi: S \to \calD$ be the blow-down of $E$ and let $\pi': S' \to \calD$ be the blow-up of the unique base point $Q$ of $|-K_{\calD}|$. In particular, the anticanonical system $|-K_{S'}|$ induces a jacobian genus one fibration $f': S' \to \bbP^1$ on $S'$. Moreover, the group $\Aut_E(S)$ is identified with the stabilizer of the image $P$ of $E$ in $\calD$. Since $Q$ is the base-point of $|-K_{\calD}|$, the $\Aut(\calD)$-action on $\calD$ lifts to $S'$ and the lifted action factors through $\Aut_0(S')$. If $m = 1$, then $P = Q$ and we get Claim (2). Hence, we may assume $m > 1$ in the following.

Note that the lifting of the $\Aut(\calD)$-action to $S'$ identifies $\Aut_E(S)$ with the stabilizer of the point $\pi'^{-1}(P)$ with respect to the action of $\Aut_0(S')$. Moreover, we have $F \in |-K_S|$ by the canonical bundle formula, hence the strict transform $F'$ of $\pi(F)$ in $S'$ is a member of the anti-canonical system and therefore $F'$ is a fiber of $f'$.

Finally, we have the following identities of normal bundles, where we use that $\pi$ and $\pi'$ restrict to isomorphisms on $F$ and $F'$, respectively, where $O = F' \cap \pi'^{-1}(Q)$, and where $m \odot \pi'^{-1}(P)$ is $m$ times the point $\pi'^{-1}(P)$ in the group law on the identity componet of $F'$ with neutral element $O$:
\begin{eqnarray*}
\calO_{F} &=& N_{F/S}^{\otimes m} = \pi^*(N_{\pi(F)/\calD} \otimes \calO_{\pi(F)}(-P))^{\otimes m} = \pi^*(\pi'_* (N_{F'/S'} \otimes \calO_{F'}(O - \pi'^{-1}(P))))^{\otimes m} \\
&=& \pi^*(\pi'_* (\calO_{F'}(mO - m\pi'^{-1}(P)))) = \pi^*(\pi'_* (\calO_{F'}(O - m \odot \pi'^{-1}(P)))).
\end{eqnarray*}
This shows that $ \pi'^{-1}(P)$ is an $m$-torsion point on the identity component of $F'$. Similarly, one shows that $\pi'^{-1}(P)$ has exact order $m$.

To finish the proof, we have to show that $g \in \Aut_E(S)$ is cohomologically trivial if and only if it acts as a cohomologically trivial automorphism on $S'$. But this is clear, since $\Pic(S) = \bbZ \cdot E \oplus \pi^*\Pic(\calD)$ and $\Pic(S') = \bbZ \cdot 
\pi'^{-1}(Q) \oplus \pi^* \Pic(\calD)$, so $g$ is cohomologically trivial on $S$ and $S'$ if and only if it acts trivially on $\Pic(\calD)$.
\end{proof}

\begin{example} 
In particular, if $f: S \to \bbP^1$ is jacobian without reducible fibers, then Theorem \ref{delPezzo} identifies $\Aut_E(S)$ with the automorphism group $\Aut(\calD)$ of a del Pezzo surface $\calD$ of degree $1$. The list of all possible finite groups acting on a del Pezzo surface of degree 1 was essentially known since the 19th century. It is contained in \cite{DolgIsk} and \cite[8.8.4]{CAG} (under an assumption on the characteristic).
\end{example}

\begin{remark} \label{rem: vectorfields}
We remark that Theorem \ref{delPezzo} can be generalized to automorphism schemes. In particular, every rational genus one surface with global vector fields is a blow-up of one of the surfaces listed in \cite[Table 6]{MartinStadlmayr}.
\end{remark}

\begin{remark} \label{rem: Halphenremark}
The situation in Theorem \ref{delPezzo} (3) can be reversed, that is, if we start with a jacobian rational genus one surface, contract a section, and blow-up a point of exact order $m$ on the image of a fiber $F'$, we get another rational genus one surface with a fiber of multiplicity $m$ and of the same type as $F'$. This process is called an \emph{Halphen transform} and we refer the reader to \cite{HarbourneLang} for further details.
\end{remark}

\section{The subgroup $\Aut_{\tr}(S)$} \label{sec: auttr}
Let $J(f):J(S)\to \bbP^1$ be the jacobian fibration of the rational genus one surface $f:S\to \bbP^1$ of index $m$. Recall that $F_0$ is a curve on $S$ such that $mF_0$ is a fiber of $f$.
Let $\MW(J(f))$ be the Mordell-Weil group of sections of $J(f)$. We consider the root lattice $L :=  F_0^\perp / \bbZ F_0 \cong \sfE_8$ and its sublattice $T$ generated by irreducible components of fibers of $f$.
Note that $\Disc(T) \cong \bigoplus_{t\in \bbP^1} \Disc_t$.
We identify $\MW(J(f))$ with the group $\Pic(S_\eta)^0$ of divisors of degree $0$ of the generic fiber $S_\eta$ defined over $K= \Bbbk(\eta)$ (i.e. the group of rational points of the jacobian variety $J_\eta$ of $S_\eta$).

\begin{lemma}
Let $\Phi: L \to \MW(J(f))$ be the restriction map. Then, $\Phi$ induces an isomorphism $L/T \cong \MW(J(f))$.
\end{lemma}

\begin{proof}
Since the kernel of $\Phi$ is $T$ by definition, it suffices to observe that $\Phi$ is surjective: Write $D \in \Pic(S_\eta)^0$ as $D = D' - D''$ with $D'$ and $D''$ effective. Then $\overline{D'} - \overline{D''} \in L$ and $\Phi(\overline{D'} - \overline{D''}) = D$.
\end{proof}

Set $M := T^\perp \subseteq L$. Since $M \cap T = \{ 0 \}$, we can consider $M$ as a subgroup of $\MW(J(f))$ via $\Phi$. For $A \in M$, we let $\tau_A \in \Aut_{\tr}(S)$ be the automorphism of $S$ induced by $\Phi(A)$.

\begin{proposition}\label{prop: gizatullin}
Let $M' \subseteq M$ be the subgroup of elements $A$ such that $\tau_A$ preserves all components of fibers of $f$. Let $d$ be the smallest positive integer such that all the $\Disc_t$ are $d$-torsion. Then,
\begin{enumerate}
\item $d \cdot \MW(J(f)) \subseteq M'$,
\item for all $A \in M'$ and $D \in \Pic(S)$, we have
\begin{equation}\label{mwaction}
\tau_A^*(D) = D + m(D.F_0) A - (\half m^2(D.F_0)A^2+mA.D)F_0 
\end{equation}
\end{enumerate}
\end{proposition}

\begin{proof}
Since the group of symmetries of the fiber over $t \in \bbP^1$ is $\Disc_t$ and multiplication by $d$ kills every $\Disc_t$, every element in $d \cdot \MW(J(f))$ preserves all components of all fibers of $f$. Hence, it suffices to show $d \cdot \MW(J(f)) \subseteq M$. Since $\Disc(T)$ is $d$-torsion and $L$ is unimodular, $\Disc(M)$ is also $d$-torsion and thus so is $\Disc(M \oplus T)$. This implies that $d \cdot L \subseteq M \oplus T$, hence $d \cdot L/T = d \cdot \MW(J(f)) \subseteq M$. This proves Claim (1).

Now, if $A \in M'$ and $D \in \Pic(S)$, then
$$
\tau_A^*(D)|_{S_\eta} = D|_{S_\eta} + {\rm deg}(D|_{S_\eta}) A|_{S_\eta} = (D + m (D.F_0) A)|_{S_\eta},
$$
hence
$$
\tau_A^*(D) = D + m(D.F_0) A + \sum a_i R_i
$$
for some $a_i$, where the $R_i$ are components of fibers of $f$. Let $R$ be a $(-2)$-curve on $S$. By definition of $M'$, we have $\tau_A^*(R) = R$ and $A.R = 0$, hence
$$
D.R = \tau_A^*(D). \tau_A^*(R) = \tau_A^*(D). R = D.R + \sum a_i R_i.R,
$$
so that $\sum a_i R_i.R = 0$. Since the radical of the span of the $R_i$ is generated by $F_0$, this implies $\sum a_i R_i = nF_0$ for some $n$. To calculate $n$, we compute
$$
D^2 = (\tau_A^*(D))^2 = D^2 + m^2 (D.F_0)^2 A^2 + 2m(D.F_0) A.D + 2n(D.F_0).
$$
Hence, we have
$$
n = - (\half m^2(D.F_0)A^2+mA.D),
$$
which is what we wanted to prove.
\end{proof}

\begin{remark}
In the case where $f$ has only irreducible fibers, that is, 
if $d = 1$, then Formula \eqref{mwaction} was first obtained by 
M. Gizatullin \cite[Proposition 9]{Gizatullin} and its proof is also 
reproduced in \cite[2.4]{CantatDolgachev}.
\end{remark}

\begin{remark}
It follows from Formula \eqref{mwaction} that the group 
$\MW(J(f))\cap \Aut_{\ct}(S)$ is $d$-torsion. We will refine this observation in Lemma \ref{lem: cohtrivialtranslations}. 
\end{remark}

\begin{corollary} \label{cor: gizatullin}
The group $d \cdot \Aut_{\tr}(S)$ is contained in the kernel $\sfE_8$ of the restriction map $W(\sfE_9) \to W(\sfE_8)$.
\end{corollary}

\begin{proof}
It suffices to observe that $d \cdot \Aut_{\tr}(S)$ acts trivially on $L$. This follows 
immediately from Formula \eqref{mwaction}, since 
$$
\tau_A^{*}(D) = D + m (A.D)F_0
$$
for all $\tau_A \in d \cdot \Aut_{\tr}(S)$ and $D \in F_0^\perp$.
\end{proof}

\begin{remark} Suppose $d = 1$ and take $\alpha\in M$ to be a root. Then $\alpha' = -F_0-\alpha$ is also a root, 
and one checks that $\tau_\alpha = (s_\alpha\circ s_{\alpha'})^m$, where $s_\beta$ is the reflection
$x\mapsto x+(x,\beta)\beta$ with respect to a root $\beta$. 
\end{remark}

\begin{remark} If $d = 1$, then Formula \eqref{mwaction} shows that $\Aut_{\tr}(S)$ is identified with the $m$-congruence subgroup $m\sfE_8$ of $W(\sfE_9)$. The reader should compare this with the corresponding result for the stabilizer of a genus one fibration on a very general Enriques or Coble surface (see e.g. \cite[p.397]{BarthPeters}).
\end{remark}

\section{The subgroup $\Aut_{\ct}(S)$}
In this section, we will study the group $\Aut_{\ct}(S)$ (whose elements are called cohomologically trivial automorphisms) for a rational genus one fibration $f: S \to \bbP^1$. We use the notation of Section \ref{Halpencil}.

\begin{lemma}\label{non-infinitely near}
Let $\pi:S \to \bbP^2$ be a realization of $S$ as a blow-up of $\bbP^2$. Assume that $\Aut_{\ct}(S)$ is non-trivial. Then, the following hold:
\begin{enumerate}
\item The action of $\Aut_{\ct}(S)$ on $S$ descends to $\bbP^2$.
\item The induced action on $\bbP^2$ preserves the irreducible components of the members of the Halphen pencil $\calP$ corresponding to $f$ that are either reducible or singular at some $q_i$.
\item If $q_1,\ldots,q_s$ are the non-infinitely near base points of $\calP$, then $s \leq 4$.
\item If $s = 4$, then we may assume that $q_1,q_2,q_3$ lie on a line $\ell$ whose strict transform $\tilde{\ell}$ on $S$ is a $(-2)$-curve. Moreover, there is a realization $\pi':S \to \bbP^2$ such that either $s \leq 3$ or $k_4 > \max\{k_1,k_2,k_3\}$.
\item One of the fibers of $f$ has at least four irreducible components.
\end{enumerate} 
\end{lemma} 

\begin{proof}
Claim (1) and (2) follow immediately from the fact that $\Aut_{\ct}(S)$ preserves all negative curves on $S$. As for Claim (3), note that by Lemma \ref{fixedpoints} at least $s-1$ of the $q_i$ lie on a line $\ell$. Thus,  $\tilde{\ell}$  has self-intersection at most $2-s$. But $\tilde{\ell}^2 \geq -2$, hence $s \leq 4$ and if $s = 4$ then $\tilde{\ell}^2 = -2$. 

As for the second part of Claim (4), assume $s = 4$  and $k_1 \leq k_2 \leq k_3$. Compose $\pi: S \to \bbP^2$ with the quadratic Cremona transformation $\iota: \bbP^2 \dashrightarrow \bbP^2$ with fundamental points $q_1,q_2,q_4$, set $\pi' := \iota \circ \pi$, and let $\calP'$ be the induced Halphen pencil. Then, either  $\calP'$ has less than four non-infinitely near base points or the multiplicity of the base point $q_3$ of $\calP'$ is $k_3 + 1$. Under the Cremona transformation we get the new base points $q_1',q_2',q_3',q_4'$ with the first three aligned and $(k_1',k_2',k_3',k_4') = (k_1,k_2,k_4-1,k_3+1)$. We may thus assume 
$$k_4 > \max \{k_1,k_2,k_3\}.$$

To prove Claim (5), observe first that if $k_i \ge 4$ for some $k_i$, then the $(-2)$-curves blown down to $q_i$ form a chain of at least $3$ curves that enter in a fiber with at least $4$ irreducible components. We may thus assume that $s \geq 3$ and $k_i < 4$ for all $i$. Moreover, as $f$ and $J(f)$ have the same types of singular fibers, we may assume that $f$ admits a section.

If $s = 4$, then the $(-2)$-curve $\tilde{\ell}$ is an irreducible component of some fiber that contains also $k= k_1+k_2+k_3-3$ irreducible components that are blown down to the points $q_1,q_2,q_3$ on the line $\ell$. Since $k_4 < 4$, we have $k \geq 3$, which yields the claim. The same argument applies if $s = 3$ and the three base points are collinear. 

If $s = 3$, the $q_i$ are not collinear, and $(k_1,k_2,k_3) = (3,3,3)$, we can apply Lemma \ref{fixedpoints} to deduce that for some $q_i$, say $q_1$, the only $\Aut_{\ct}(S)$-invariant lines through $q_1$ are $\ell_{1i} := \langle q_1,q_i \rangle$, $i = 2,3$. Then, we may assume that the tangent line to a general member of the pencil at $q_1$ equals $\ell_{12}$, so that the strict transform $\til_{12}$ of $\ell_{12}$ on $S$ is a $(-2)$-curve. The fiber of $f$ containing $\til_{12}$ has at least $4$ components. 
%
\end{proof}

 \subsection{Cohomologically trivial automorphisms of jacobian rational genus one surfaces} \label{sec: cohtrivialjacobian}
In this section, we calculate the group $\Aut_{\ct}(S)$ for jacobian rational genus one surfaces $f:J \to \bbP^1$. 

\begin{proposition} \label{equations}
Let $f: J \to \bbP^1$ be a jacobian rational genus one surface. If the singular fibers of $f$ are as in Column $2$ of Table \ref{Table2} and either $f$ is extremal or $\Aut_{\ct}(J)$ is as in Column $5$ and $6$ of Table \ref{Table2}, then $f$ admits one of the Weierstrass equations in Column $4$. Conversely, a surface with the given Weierstrass equation admits the group $\Aut_{\ct}(J)$ and the number of moduli given in Column $5$,$6$, and $7$.
\end{proposition}
\begin{proof}
For the first statement, note that the equations for the extremal fibrations are well-known and we refer the reader to \cite{MirandaPersson}, \cite{Langextremal1}, \cite{Langextremal2}, \cite{Ito1}, and \cite{Ito2}. For the non-extremal fibrations, it follows from the Table that $\Aut_{\ct}(J)^\dagger = \{ \id \}$. Then, in characteristic $p \ne 2,3$, finding the Weierstrass equations is a simple exercise using Tate's algorithm while guaranteeing that the discriminant of $f$ admits the stated symmetries. The only case where this approach is complicated is Case $22$. In this case, note that $J/\Aut_{\ct}(J)$ is an elliptic surface with singular fibers $\tilde{D}_4,\tilde{A}_3,\tilde{A}_0,\tilde{A}_0$. This fibration admits a $2$-torsion section by \cite[Table 8.2]{ShiodaSchutt}, hence so does $j$ and we can find a Weierstrass equation of the form $y^2 = x^3 + a_2x^2 + a_4 x$. From there, the computation is straightforward.
For the Weierstrass equations in the case $p = 2,3$, one can simplify the equations given in \cite{Langextremal1} and \cite{Langextremal2}. We leave the details to the reader.

Let us prove the converse statement. The group $\Aut_{\ct}(J)^\dagger$ is a subgroup of the automorphism group of the generic fiber of $j$ and the action of this group on $\MW(J(f))$ is well-known, which makes the calculation of $\Aut_{\ct}(J)^\dagger$ straightforward. For the calculation of $\Aut_{\ct}(J)/\Aut_{\ct}(J)^\dagger$ we describe some automorphisms of $J$ in Table \ref{automorphismsofequations} and leave it to the reader to check that these automorphisms exhaust all elements in $\Aut_{\ct}(J)$. In Table \ref{automorphismsofequations}, we have $a \in \bbG_m$, $b \in \bbG_a$ and $\zeta_3^3 = 1$ and $i =6,4,3,2$, respectively.
\begin{table}
$$
\begin{array}{|l|c|c|c|c|c|c|} \hline
\text{Case} & t & x & y &  \multicolumn{2}{c|}{  \text{Action on fiber at } t = \infty }\\ \hline
5,8 & a^6t + b^p & a^2x + 2b & a^3y + 3b &x \mapsto a^{-10}x &  y \mapsto a^{-15}y
 \\
4,7 & t+ b^p + b & x+2b & y + 3b & x \mapsto x & y \mapsto y \\
10 & t+ b^2 + b & x & y + bx & x \mapsto x & y \mapsto y \\
11 & t + b^2 & x & y + bx + b & x \mapsto x &  y \mapsto y \\
1,13,23,31 & a^it & a^2x & a^3 y &  x \mapsto a^{2 - 2i}x & y \mapsto a^{3 - 3i}y
  \\
2,9 & -t & -x & \sqrt{-1}y & x \mapsto -x & y \mapsto -\sqrt{-1}y  \\
21,22,24,25,26,27,28,29 & -t & x & -y & x \mapsto x & y \mapsto y  \\
12,15,18 & \zeta_3 t & \zeta_3^2 x & y & x \mapsto x & y \mapsto y  \\
19,20 & \zeta_3 t & \zeta_3 x & y & x \mapsto \zeta_3^2 x & y \mapsto y \\
17 & t + \lambda & x & y + x + \lambda & x \mapsto x & y \mapsto y \\ \hline
\end{array}
$$
\caption{Automorphisms in $\Aut_{\ct}(J) \setminus \Aut_{\ct}(J)^\dagger$}
\label{automorphismsofequations}
\end{table}
\end{proof}

The following exceptional case, which corresponds to Case 17 in Table \ref{Table2} will occur in the proof of Theorem \ref{thm: main}.

\begin{lemma} \label{lem: exception17}
Let $f: S \to \bbP^1$ be a jacobian rational elliptic surface whose only reducible fiber $F_0$ is of type $\tilde{E}_7$ and which admits an automorphism $g \in \Aut_{\ct}(S)$ with $\ord(g) = p$. Then, $p = 2$, the singular fibers of $f$ are $F_0$ and two fibers of type $\tilde{A}_0^*$, $\Aut_{\ct}(S) = \bbZ/2\bbZ$, and $\Aut_{\ct}(S)^\dagger = \{{\rm id}\}$.
\end{lemma}

\begin{proof}
First, we claim that $\Aut_{\ct}(S)^\dagger = \{{\rm id}\}$. Indeed, all fibers of $f$ are $\Aut_{\ct}(S)^\dagger$-invariant. Since $f$ is not extremal and $\Aut_{\ct}(S)^\dagger$ preserves all its sections, we have infinitely many fixed points on each fiber, , so $\Aut_{\ct}(S)^\dagger = \{{\rm id}\}$. In particular, $g \not \in \Aut_{\ct}(S)^\dagger$. 

If $p \neq 2$, then Lemma \ref{lem: somefibers} shows that $f$ admits three fibers of type $\tilde{A}_0^*$, so we have $p = 3$. By \cite{Lang3} the discriminant of $f$ is of the form
$$
\Delta = t^9 s(c_1^3 - tc_0^2 c_1^2 + t^2 c_0^3 d_1)
$$
for homogeneous polynomials $c_i,d_i \in k[s,t]$ of degree $i$. The fiber $F_0$ is located at $t = 0$ and it is easy to check that there is no linear automorphism of $k[s,t]$ that permutes the other zeroes of $\Delta$, so $g$ does not exist. Hence, we must have $p = 2$. By Lemma \ref{lem: somefibers}, we only have to exclude the case where $F_0$ is the only singular fiber of $f$. In this case, by \cite[Table 2]{Stadlmayr}, $f$ is given by the Weierstrass equation
$$
y^2 + y = x^3 + tx
$$
and one can easily check that it admits no involutions acting non-trivially on the base.
\end{proof}

\begin{theorem} \label{thm: main}
Let $f: S \to \bbP^1$ be a jacobian rational elliptic surface. Assume that $\Aut_{\ct}(S)$ is non-trivial. Then, the singular fibers of $f$ are as in Table \ref{Table2}. Moreover, either $f$ is extremal or the groups $\Aut_{\ct}(S)$ and $\Aut_{\ct}(S)^\dagger$ are as in Column $5$ and $6$ Table \ref{Table2}.
\end{theorem}
\begin{table}[h!]
\scalebox{0.8}{
$\displaystyle
\begin{array}{|l|c|c|c|c|c|c|c|}
\hline
&\text{Singular fibers} & e \text{ or } qe & \text{Weierstrass equation} & \Aut_{ct}(J)^\dagger & \Aut_{ct}(J) &  \text{Moduli} & p \\ \hline
1&\tilde{E}_8,\tilde{A}_0^{**} & e & y^2 = x^3 + t & \bbZ/6\bbZ & \bbG_m &0 & \neq 2,3\\
2&\tilde{E}_8, \tilde{A}_0^*, \tilde{A}_0^* & e & y^2 = x^3 + x + t &  \bbZ/2\bbZ & \bbZ/4\bbZ & 0 & \neq 2,3 \\
3&\tilde{E}_8, \tilde{A}_0^* & e & y^2 = x^3 + x^2 + t
&
\bbZ/2\bbZ & \bbZ/2\bbZ & 0 & 3 \\
4&\tilde{E}_8 & e & y^2 = x^3 + x + t & \bbZ/6\bbZ & \bbZ/2\bbZ \times \bbG_a & 0 & 3 \\
5&\tilde{E}_8 & qe & y^2 = x^3 + t & \bbZ/2\bbZ & {\rm Aff} & 0 & 3  \\ 
6&\tilde{E}_8, \tilde{A}_0^* & e & 
y^2 + xy = x^3 + t
&
\bbZ/2\bbZ & \bbZ/2\bbZ & 0 &2  \\
7&\tilde{E}_8 & e & y^2 + y = x^3 + t&Q_8 & (\bbZ/2\bbZ)^2 \cdot \bbG_a & 0 & 2 \\
8&\tilde{E}_8 & qe & y^2 = x^3 + t & \bbZ/3\bbZ & {\rm Aff} & 0 & 2 \\ \hline
9&\tilde{D}_8, \tilde{A}_0^*, \tilde{A}_0^* & e & y^2 = x^3 + tx^2 + x & \bbZ/2\bbZ & \bbZ/4\bbZ & 0 & \neq 2 \\
10&\tilde{D}_8 & e & y^2 + xy = x^3 + tx^2 + \lambda, \lambda \neq 0 & \bbZ/2\bbZ & \bbG_a & 1 & 2\\
11&\tilde{D}_8 & qe & y^2 = x^3 + tx^2 + t & \{\id\} & \bbG_a & 0 & 2 \\ \hline
12&\tilde{A}_8, \tilde{A}_0^*,\tilde{A}_0^*,\tilde{A}_0^* & e & y^2 + txy - y = x^3 & \{\id\} & \bbZ/3\bbZ  & 0 & \neq 3 \\ \hline
13&\tilde{E}_7,\tilde{A}_1^{*} & e/qe & y^2 = x^3 + tx & \mu_4 & \bbG_m& 0 & any \\
14&\tilde{E}_7, \tilde{A}_1, \tilde{A}_0^* & e & y^2 = x^3 +x^2+ tx &  \bbZ/2\bbZ &  \bbZ/2\bbZ & 0 & \neq 2   \\
15&\tilde{E}_7,\tilde{A}_0^{*},\tilde{A}_0^{*},\tilde{A}_0^{*} & e& y^2 = x^3 + tx + 1 & \{\id\}& \bbZ/3\bbZ& 0 &\ne 2,3 \\
16&\tilde{E}_7, \tilde{A}_1 & e & y^2  + xy = x^3 + tx &  \bbZ/2\bbZ & \bbZ/2\bbZ & 0 & 2  \\
17&\tilde{E}_7,\tilde{A}_0^{*},\tilde{A}_0^{*}& e& y^2 + xy = x^3 + tx + \lambda t, \lambda \neq 0 & \{\id\}& \bbZ/2\bbZ & 1 &2  \\
18&\tilde{E}_7 & e & y^2 + y = x^3 + tx & \{\id\} &  \bbZ/3\bbZ  & 0 & 2 \\ \hline
19&\tilde{D}_7, \tilde{A}_0^*,\tilde{A}_0^*,\tilde{A}_0^* & e & y^2 = x^3 + tx^2 + 1 & \{\id\} & \bbZ/3\bbZ & 0 & \neq 2,3  \\
20&\tilde{D}_7 & e & y^2 + y = x^3 + tx^2 &\{\id\} & \bbZ/3\bbZ& 0 & 2   \\ \hline
21&\tilde{A}_7, \tilde{A}_1, \tilde{A}_0^*, \tilde{A}_0^* & e & y^2 = x^3 + 2(t^2 - 1)x^2 - x& \{\id\} & \bbZ/2\bbZ  & 0 & \neq 2  \\
22&\tilde{A}_7, \tilde{A}_0^*, \tilde{A}_0^*, \tilde{A}_0^*, \tilde{A}_0^* & e & y^2 = x^3 + 2(t^2 - 1)x^2 - \lambda x, \lambda \neq 0,1 & \{\id\} & \bbZ/2\bbZ  & 1 & \neq 2 \\ \hline
23&\tilde{E}_6, \tilde{A}_2^* & e/qe & y^2 + ty = x^3 &  \mu_3 & \bbG_m & 0 & \text{any} \\
24&\tilde{E}_6, \tilde{A}_0^*, \tilde{A}_0^*, \tilde{A}_0^*, \tilde{A}_0^* & e  & y^2 = x^3 -3x + t^2 + \lambda, \lambda \neq 0,2 & \{\id\} & \bbZ/2\bbZ  & 1 & \neq 2,3  \\
25&\tilde{E}_6,\tilde{A}_0^{**},\tilde{A}_0^{**} & e & y^2 = x^3 -3x + t^2 + 2 & \{\id\} & \bbZ/2\bbZ  & 0 & \neq 2,3  \\
26&\tilde{E}_6,\tilde{A}_1,\tilde{A}_0^*,\tilde{A}_0^* & e & y^2 = x^3 -3x + t^2 & \{\id\} & \bbZ/2\bbZ  & 0 & \neq 2,3  \\
27&\tilde{E}_6, \tilde{A}_0^*, \tilde{A}_0^* & e & y^2 = x^3 + x^2 + t^2 + \lambda, \lambda \neq 0 & \{\id\} & \bbZ/2\bbZ  & 1 & 3 \\
28&\tilde{E}_6, \tilde{A}_1  & e & y^2 = x^3 + x^2 + t^2 & \{\id\} & \bbZ/2\bbZ  & 0 & 3 \\
29&\tilde{E}_6 & e & y^2 = x^3 + x + t^2  &  \{\id\} & \bbZ/2\bbZ  & 0 & 3 \\ \hline
30&\tilde{D}_6,  \tilde{A}_1, \tilde{A}_1 & e & y^2 = x^3 + tx^2 + x + t & \bbZ/2\bbZ & \bbZ/2\bbZ & 0 & \neq 2 \\ \hline
31&\tilde{D}_4,  \tilde{D}_4 & e/qe & y^2 = x^3 + \lambda tx^2 + t^2x, \lambda \neq 2& \mu_2 & \bbG_m & 1 & any \\ \hline
\end{array}
$}
\bigskip
\caption{Weierstrass equations for jacobian rational elliptic surfaces with cohomologically trivial automorphisms}
\label{Table2}
\end{table}

\begin{proof}
We fix some realization $\pi:S \to \bbP^2$ as in Section \ref{Halpencil} and use the same notation as in that section. By Lemma \ref{non-infinitely near}, there is a reducible fiber $F_0$ of $f$ with at least $4$ components and this $F_0$ is necessarily $\Aut_{\ct}(J)$-invariant. We set $C_0 = f(F_0)$ and let $g \in \Aut_{\ct}(J)$ be a non-trivial automorphism of order $n$. 

\underline{Case: $s = 1$} 

In this case, $F_0$ contains a chain of $8$ $(-2)$-curves, so it is of type $\tilde{A}_8$ or $\tilde{E}_8$. This leads precisely to the \textbf{Cases 1,$\hdots$,8, and 12}.  


In the remaining cases, we will write $C_0$ as in the following claim.

\noindent \textit{Claim: If $s \geq 2$, then $C_0 = \ell + Q$ for a line $\ell$ and a conic $Q$}

\noindent If $C_0$ is not of the stated form, then it is integral and singular. We may assume that the singularity of $C_0$ is at $q_1$.  Then, $F_0 = \tilde{C}_0 + R_1^{(1)} + \hdots + R_1^{(k_1-1)}$, so $C_0$ cannot be cuspidal. Since $k_1 \geq 3$, the tangent lines to $F_0$ at $q_1$ are $g$-invariant, so they are not swapped by $g$ and hence $g$ has only one fixed point on $C_0$ by Lemma \ref{autoelliptic}. Since the $q_i$ are fixed points of $g$, this is impossible if $s \geq 2$. \hfill $\blacksquare$

\underline{Case: $s = 2$}

Let $\ell_{12} = \langle q_1,q_2 \rangle$.

First, assume that $Q$ is smooth and that $\ell \cap Q$ consists of two distinct points $q,q'$. Since $F_0$ has more than $3$ components, we may assume $q = q_2$. If $q_1 \not \in \ell$, then $q',q_1,q_2 \in Q$ are $3$ fixed points of $g$, which is impossible by Lemma \ref{fixedpoints}. So, $\ell = \ell_{12}$. We have the following subcases:

\begin{itemize}
\item
 If $q' = q_1$, then all $R_{i}^{(j)}$ with $j < k_i$ are components of $F_0$, so $F_0$ is of type $\tilde{A}_8$. This is \textbf{Case 12}.

\item If $q' \neq q_1$, then $(k_1,k_2) = (2,7)$, since we have to blow up $6$ points on $Q$ and $3$ points on $\ell$. Then, $F_0$ is of type $\tilde{A}_7$ and there is another reducible fiber of $f$ containing $R_1^{(1)}$. This is \textbf{Case 21}.
\end{itemize}

Next, assume that $Q$ is smooth and $\ell \cap Q$ consists of a single point $Q$. As above, we may assume that $q = q_2$. We have the following subcases:
\begin{itemize}
\item If $\ell \neq \ell_{12}$, then $q_1 \in Q$, $(k_1,k_2) = (4,5)$ and $F_0$ is of type $\tilde{D}_5$. On the one hand, $g$ fixes the two adjacent components of multiplicity $2$ of $F_0$ pointwise, so $(p,n) \neq 1$, and on the other hand, $g$ acts non-trivially on $Q$, so $(p,n) = 1$. This contradiction shows that this case does not occur.
\item If $\ell = \ell_{12}$, then $(k_1,k_2) = (1,8)$ and $F_0$ is of type $\tilde{D}_8$. This yields \textbf{Cases 9, 10, and 11}.
\end{itemize}

Now, assume that $C_0 = \ell + \ell_1 + \ell_2$ for distinct lines $\ell_1$ and $\ell_2$. It is easy to check that the lines cannot all be distinct, so we may assume $\ell = \ell_1$. Possibly interchanging the $\ell_i$, we may assume that $q_i \in \ell_i$.
\begin{itemize}
\item If $\ell_{12} \not \in \{\ell,\ell_1\}$, then $(k_1,k_2) = (6,3)$ and $F_0$ is of type $\tilde{E}_6$. There is another fiber $F_1$ of $f$ containing $R_2^{(1)} + R_2^{(2)}$. Hence, this is \textbf{Case 23}.
\item If $\ell_1 = \ell_{12}$, then $(k_1,k_2) = (4,5)$ and $F_0$ is of type $\tilde{D}_8$. This leads to \textbf{Cases 9, 10, and 11}.
\item If $\ell_2 = \ell_{12}$, then $(k_1,k_2) = (7,2)$ and $F_0$ is of type $\tilde{E}_7$.  There is another fiber $F_1$ of $f$ containing $R_2^{(1)}$. Hence, this leads to \textbf{Cases 13, 14, and 16}.
\end{itemize}

Finally, assume that $C_0 = 3\ell_{12}$. A general member of $f$ meets $C_0$ at the $q_i$ with multiplicities $3$ and $6$, so we may assume $(k_1,k_2) = (3,6)$ and $x_2^{(2)} \in C_0$. Then, $F_0$ is of type $\tilde{E}_7$. Let $t_1$ be the line with $x_1^{(2)} \in t_1$ and let $Q_1$ be the conic with $x_2^{(1)},\hdots,x_2^{(5)} \in Q_1$. We have the following subcases:
\begin{itemize}
\item If $x_1^{(3)} \in t_1$, then $\tilde{t}_1$ is a $(-2)$-curve and we arrive at the \textbf{Cases 13, 14, and 16}. A similar argument works if $x_2^{(6)} \in Q_1$.
\item If $x_1^{(3)} \not \in t_1$, $x_2^{(6)} \not \in Q_1$ and $t_1 \cap Q_1 = \{q\}$ for a single point $q$, then $q$ is a fixed point of $g$, and $\tilde{t}_1$ and $\tilde{Q}_1$ are $(-1)$-curves. By Lemma \ref{fixedpoints}, we have $(p,n) = 1$. Moreover, note that $g$ acts non-trivially on $Q_1$, so $\Aut_{\ct}(S)^\dagger = \{{\rm id}\}$.
Let $F_1$ be the second $g$-invariant fiber of $f$ and note that $F_1$ is irreducible since $\pi(F_1)$ does not contain any line whose proper transform is a $(-2)$-curve. Since $f$ admits the $3$ disjoint sections $\tilde{t}_1,E_1,$ and $E_2$, Lemma \ref{autoelliptic} shows that $F_1$ has to be smooth and since $q,q_1,q_2 \in F_1$, we have $n \leq 3$. If $p \neq 2$, then Lemma \ref{lem: somefibers} shows that $f$ admits $3$ fibers of type $\tilde{A}_0^*$, so $n =3$. This is \textbf{Case 15}. If $p = 2$, then $n = 3$ is implied by $(p,n) = 1$. Then, Lemma \ref{lem: somefibers} shows that $F_0$ is the only singular fiber of $f$, so this is \textbf{Case 18}.

\item If $x_1^{(3)} \not \in t_1$, $x_2^{(6)} \not \in Q_1$ and $t_1 \cap Q_1 = \{q,q'\}$ for two distinct points $q,q'$, then $\tilde{t}_1$ and $\tilde{Q}_1$ are again $(-1)$-curves. Then, $g$ swaps $q$ and $q'$, for otherwise it would fix $t_1$ and $Q_1$ pointwise. In particular, we have $\Aut_{\ct}(S)^\dagger = \{{\rm id}\}$, and $\Aut_{\ct}(S) = \bbZ/2\bbZ$. If $p \neq 2$, then $g$ has another fixed point $q'' \in t_1$ and, by Lemma \ref{fixedpoints}, $g$ fixes one of the lines $\ell_{12},t_1,$ and $\langle q'',q_2 \rangle$ pointwise. In every case, this would imply that $g$ fixes $Q_1$ pointwise, which is absurd. Hence, $p = 2$ and we can apply Lemma \ref{lem: exception17} to deduce that we are in \textbf{Case 17}.
\end{itemize}

\underline{Case: $s = 3$ and $q_1,q_2,q_3$ collinear}

Let $\ell_{123}$ be the line through the $q_i$.

First, assume that $Q$ is smooth and $\ell \cap Q$ consists of two distinct points $q,q'$. Since $F_0$ has more than $3$ components, we may assume $q = q_3$. Note that $Q$ can contain at most $2$ of the $q_i$ by Lemma \ref{fixedpoints}, so we may assume $\ell = \ell_{123}$, $q_1 \not \in Q$ and $q_2 \in Q$. In particular, we have $k_1 = 1$. The $R_i^{(j)}$ with $i \in \{2,3\}$ and $j < k_i$ are in $F_0$, so $F_0$ is of type $\tilde{A}_7$. Let $t_i$ be the tangent line to $Q$ at $q_i$. We have the following two subcases:
\begin{itemize}
\item If $x_i^{(3)} \in t_i$ for some $i \in \{2,3\}$, then $\tilde{t}_i$ is a $(-2)$-curve, so $f$ admits another reducible fiber apart from $F_0$. Thus, we are in \textbf{Case 21}.
\item If $x_i^{(3)} \not \in t_i$ for $i = 2,3$, then $\tilde{t}_i$ is a $(-1)$-curve on which $g$ acts non-trivially by Lemma \ref{fixedpoints}. So, we have $\Aut_{\ct}(S)^\dagger = \{{\rm id}\}$. Let $F_1$ be the fiber of $f$ such that $C_1 := \pi(F_1)$ passes through the intersection $\tilde{t}_2 \cap \tilde{t}_3$. This $C_1$ is $g$-invariant, irreducible, and contains $4$ fixed points of $g$, so by Lemma \ref{autoelliptic} it is smooth and $\Aut_{\ct}(S) = \bbZ/2\bbZ$. This is \textbf{Case 22}.
\end{itemize}

Next, assume that $Q$ is smooth and $\ell \cap Q$ consists of a single point $q$. As in the previous case, we may assume $q = q_3$ and $\ell = \ell_{123}$. Then, we have $(k_1,k_2,k_3) = (1,1,7)$, but since $Q$ and $\ell$ are tangent at $q_3$, we have $\til^2 \leq -3$, which is impossible. So, this case does not occur.

Assume that $C_0 = \ell_1 + \ell_2 + \ell_3$ is the sum of three distinct lines. Since $F_0$ has more than $3$ components, one of the lines has to be $\ell_{123}$, so we may assume $\ell_1 = \ell_{123}$. It is easy to check that each of the $\ell_i$ has to contain some $q_j$, so we may assume $q_i \in \ell_i$. Then, we have $(k_1,k_2,k_3) = (1,4,4)$ and $F_0$ is of type $\tilde{A}_8$. This is \textbf{Case 12}.

Now, assume that $C_0 = \ell_{123} + 2\ell$ for some line $\ell \neq \ell_{123}$. We may assume $q_3 \in \ell_{123}$. Then, $(k_1,k_2,k_3) = (1,1,7)$ and $F_0$ is of type $\tilde{E}_7$. Since $g$ acts non-trivially on $E_1$, we have $\Aut_{\ct}(S)^\dagger = \{{\rm id}\}$. If $(p,n) = 1$, then $g$ admits another fixed point $q \in \ell \setminus \ell_{123}$. Let $F_1 \neq F_0$ be a $g$-invariant fiber and note that $C_1 := \pi(F_1)$ is irreducible, since the only lines whose proper transforms are $(-2)$-curves are $\ell_{123}$ and $\ell$. 
Since $q \in \ell \setminus \ell_{123}$, $F_1$ does not contain $q$, so $g$ has precisely the $3$ smooth fixed points $q_i$ on $F_1$. By Lemma \ref{autoelliptic}, this implies that $F_1$ is smooth and $n = 3$, so these are \textbf{Cases 15 and 18}. If $(p,n) \neq 1$, we can apply Lemma \ref{lem: exception17} to deduce that we are in \textbf{Case 17}. 

Next, assume that $C_0 = 2\ell_{123} + \ell$ for some line $\ell \neq \ell_{123}$ with $q_3 \in \ell$. Then, we have $(k_1,k_2,k_3) = (2,2,5)$ and $F_0$ is of type $\tilde{D}_7$. Let $t_i$ with $i = 1,2$ be the line through $q_i$ with $x_i^{(2)} \in t_i$. Then, $t_1,t_2,$ and $\ell$ meet in a single point $q$ which is a fixed point of $g$, so $(p,n) = 1$ by Lemma \ref{fixedpoints}. Since $\tilde{t}_1$ is a $(-1)$-curve and $g$ acts non-trivially on it, we have $\Aut_{\ct}(S)^\dagger = \{{\rm id}\}$. Let $F_1 \neq F_0$ be a $g$-invariant fiber. By the same argument as in the previous paragraph, $F_1$ is smooth and $n = 3$, so this leads to \textbf{Cases 19 and 20}.

Finally, assume that $C_0 = 3\ell_{123}$. We must have $(k_1,k_2,k_3) = (3,3,3)$ and then $F_0$ is of type $\tilde{E}_6$. We let $t_i$ be the line through $q_i$ with $x_i^{(2)}$. These lines meet in a single point $q$ which is a fixed point of $g$, so $(p,n) = 1$ by Lemma \ref{fixedpoints}.
If for some $i$ we have $x_i^{(3)} \in t_i$, then $\tilde{t}_i$ is a $(-1)$-curve on which $g$ acts non-trivially, so $\Aut_{\ct}(S)^\dagger = \{{\rm id}\}$. We have the following subcases:
\begin{itemize}
\item If $x_i^{(3)} \in t_i$ for two of the $i$, then the corresponding $\tilde{t}_i$ are part of the same fiber of $f$. This is only possible if all three of the $\tilde{t}_i$ are $(-2)$-curves and then we are in \textbf{Case 23}.
\item If $x_i^{(3)} \in t_i$ for precisely one $i$, say for $i = 1$, then the fiber $F_1$ of $f$ containing $\tilde{t}_1$ also contains the strict transform of a smooth conic $Q_{23}$ through $q_2$ and $q_3$. Since $g$ acts non-trivially on $Q_{23}$, we have $Q_{23} \cap t_1$ consists of $2$ points, both different from $q$. Hence, we have $n = 2$ and we obtain \textbf{Cases 26 and 28}.
\item If $x_i^{(3)} \not \in t_i$ for all $i$, let $F_1$ be the fiber of $f$ passing through $q$ and let $C_1 = \pi(F_1)$. Note that $C_1$ is irreducible, since the only line whose proper transform is a $(-2)$-curve is $\ell_{123}$. Then, $C_1$ contains $4$ fixed points of $g$, so by Lemma \ref{autoelliptic} $C_1$ is smooth and $n = 2$. Thus, this leads to the \textbf{Cases 24, 25, 27, and 29}. 
\end{itemize}

\underline{Case: $s = 3$ and $q_1,q_2,q_3$ are not collinear}

Let  $\ell_{ij} = \la q_i,q_j\ra$ and note that, by Lemma \ref{fixedpoints}, we have $(p,n) = 1$ .

First, assume that $Q$ is smooth. Since $F_0$ has more than $3$ components, we may assume $q = q_3$. Note that $Q$ can contain at most $2$ of the $q_i$ by Lemma \ref{fixedpoints}, so we may assume $q_1 \in \ell \setminus Q$. Then, $\ell = \ell_{13}$ and $Q \cap \ell = \{q_3\}$. Since $F_0$ has more than $3$ components, we have $k_3 \geq 4$ and $F_0$ is of type $\tilde{D}_{k_3}$. Also, as $Q$ and $\ell$ are tangent at $q_3$, this implies $k_1 = 1$. Note that the tangent to $Q$ at $q_2$ is $g$-invariant, so it coincides with $\ell_{12}$, for otherwise it would produce a third fixed point on $\ell_{13}$ forcing $\ell_{13}$ and $Q$ to be fixed pointwise.
So, we have the following subcases:

\begin{itemize}
\item If $k_2 \geq 2$, then $\til_{12}$ is a $(-2)$-curve and there exists a fiber $F_1$ of $f$ and a smooth conic $Q'$ such that $C_1 = \pi(F_1) =  \ell_{12} + Q'$. By the same argument as in the previous paragraph, $\ell_{12}$ is tangent to $Q'$ at $q_2$. Computing the proper transforms of $Q'$ and $\ell_{12}$, we see that $F_1$ can only be a fiber of $f$ if $(k_1,k_2,k_3) = (1,4,4)$, in which case both $F_0$ and $F_1$ are of type $\tilde{D}_4$. This is \textbf{Case 31}.
\item If $(k_1,k_2,k_3) = (1,1,7)$, note that $g$ acts non-trivially on one of the sections $\til_{12}$ or $\til_{23}$, so $\Aut_{\ct}(S)^\dagger = \{{\rm id}\}$. Let $F_1$ be the second $g$-invariant fiber and set $C_1 := \pi(F_1)$. Note that $C_1$ is irreducible, since the only line that becomes a $(-2)$-curve on $S$ is $\ell_{13}$. None of the $\ell_{ij}$ is fixed pointwise by $g$, so $n \neq 2$ by Lemma \ref{fixedpoints} and $C_1$ contains precisely $3$ fixed points. Thus, Lemma \ref{autoelliptic} shows that $C_1$ is smooth and $n = 3$. This leads to \textbf{Cases 19 and 20}. 
\end{itemize}

Next, assume that $C_0$ consists of $3$ distinct lines. Since $F_0$ has at least $4$ components, at least two of these lines have to coincide with some $\ell_{ij}$, so we may assume $C_0 = \ell_{12} + \ell_{23} + \ell$ for some line $\ell$. We have the following subcases:

\begin{itemize}
\item If $\ell = \ell_{13}$, we may assume that a general member of the pencil is tangent to $\ell_{12}$ at $q_1$, and then it must be tangent to $\ell_{23}$ at $q_3$ and tangent to $\ell_{12}$ at $q_2$. We have the following picture.

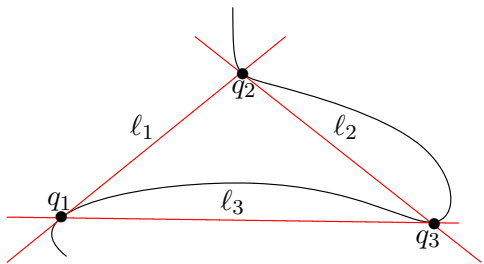
\xy (-55,25)*{};
(4,-5)*+{};(25,30)*+{}
**\crv{(-7,2)&(35,10)&(55,-5)&(55,12)&(25,20)&};
(-5,1)*{};(55,.2)*{}**{\color{red}\dir{-}};(-5,-5)*{};(32,25)*{}**{\color{red}\dir{-}};(58,-5)*{};(20,25)*{}**{\color{red}\dir{-}};
(13,13)*{\ell_{1}};(40,13)*{\ell_{2}};(25,3)*{\ell_{3}};(2,3)*{q_1};(51,-2)*{q_3};(26.5,18)*{q_2};(2.2,1)*{\bullet};(51.8,0)*{\bullet};(26.3,20)*{\bullet};
\endxy
\noindent It is immediate to see that $F_0$ is of type $\tilde{A}_8$ and thus this is \textbf{Case 12}.
\item If $\ell \neq \ell_{13}$ and $q_2 \in \ell$, then $(k_1,k_2,k_3) = (2,5,2)$ and $F_0$ is of type $\tilde{D}_6$. Moreover, there are two additional reducible fibers of $f$ containing $R_1^{(1)}$ and $R_3^{(1)}$, respectively. Lemma \ref{lem: somefibers} shows that we are in \textbf{Case 30}.
\item If $\ell \neq \ell_{13}$ and $q_2 \not \in \ell$, then we may assume $q_3 \in \ell$. Then, we have $(k_1,k_2,k_3) = (2,3,4)$ and $F_0$ is of type $\tilde{A}_7$. There is another reducible fiber $F_1$ containing $R_1^{(1)}$, which is necessarily of type $\tilde{A}_1$. Thus, this is \textbf{Case 21}.
\end{itemize}

Now, assume that $C_0 = 2\ell_{12} + \ell$ for some line $\ell \neq \ell_{12}$. Clearly, we must have $q_3 \in \ell$. If $\ell = \ell_{ij}$ for some $i$ and $j$, then we assume $\ell = \ell_{23}$. We have the following subcases:
\begin{itemize}
\item If $\ell \neq \ell_{ij}$ for all $i$ and $j$, then $\ell$ meets $\ell_{12}$ at some point $q \neq q_1,q_2$ and thus $g$ fixes $\ell_{12}$ pointwise. Moreover, we have $k_3 = 3$. Since $C_0$ has multiplicity $2$ at $q_1$ and $q_2$, we may assume that $(k_1,k_2,k_3) = (4,2,3)$. Then, $x_1^{(2)} \in \ell_{12}$ and $x_3^{(2)} \in \ell_{23}$, since the latter is the unique $g$-invariant line different from $\ell_{12}$ through $q_2$. So, $F_0$ is of type $\tilde{D}_5$. Since $(p,n) = 1$ and $g$ fixes both $\til_{12}$ and $R_1^{(2)}$ pointwise, this case does not occur. 
\item If $\ell = \ell_{23}$ and $k_3 = 1$, then we must have $k_1,k_2 \geq 2$, $x_1^{(2)} \in \ell_{12}$, and $x_2^{(2)} \in \ell_{23}$. Computing the proper transform of $C_0$, we get $(k_1,k_2,k_3) = (4,4,1)$ and $F_0$ is of type $\tilde{D}_7$. As in the case above where $Q$ is smooth and $(k_1,k_2,k_3) = (1,1,7)$, one can show that $n = 3$ and $\Aut_{\ct}(S)^\dagger = \{{\rm id}\}$, so these are \textbf{Cases 19 and 20}.

\item If $\ell = \ell_{23}$, $k_3 = 2$, and $x_1^{(2)} \in \ell'$ for some line $\ell' \neq \ell_{12}, \ell_{13}$, then $(k_1,k_2,k_3) = (2,5,2)$ and $F_0$ is of type $\tilde{E}_6$. Since $g$ acts non-trivially on the $(-1)$-curve $\til'$ we have $\Aut_{\ct}(S)^\dagger = \{{\rm id}\}$. Let $F_1$ be the fiber of $f$ containing. Since the only lines that become $(-2)$-curves on $S$ are $\ell_{12}$ and $\ell_{13}$, $C_1 = \pi(F_1)$ is irreducible and by Lemma \ref{autoelliptic} $C_1$ is nodal and $n = 2$, so $F_1$ is of type $\tilde{A}_1$. This leads to \textbf{Cases 26 and 28}.

\item If $\ell = \ell_{23}$, $k_3 = 2$, and $x_1^{(2)} \in \ell_{13}$, then again $(k_1,k_2,k_3) = (2,5,2)$ and $F_0$ is of type $\tilde{E}_6$. This time, $F_1$ also contains $\til_{23}$, so it consists of three components. This is \textbf{Case 23}.

\item If $\ell = \ell_{23}$, $k_3 = 2$, and $x_1^{(2)} \in \ell_{12}$, then $(k_1,k_2,k_3) = (4,3,2)$ and $F_0$ is of type $\tilde{D}_6$. Additionally, there are two distinct fibers containing $R_1^{(1)}$ and $R_3^{(1)}$, respectively. Thus, this is \textbf{Case 30}. 
\end{itemize}

Finally, assume that $C_0 = \ell_{12} + 2\ell$ for some line $\ell$ different from the $\ell_{ij}$. Then, we must have $q_3 \in \ell$ and we may assume $(k_1,k_2,k_3) = (1,2,6)$. The fiber $F_0$ is of type $\tilde{E}_6$ and there is another reducible fiber $F_1$ of $f$ containing $R_2^{(1)}$. The only $g$-invariant lines that become $(-2)$-curves on $S$ are $\ell_{12}$ and $\ell'$, so $C_1 = \pi(F_1)$ is irreducible. Thus, we have $n = 2$ and $C_1$ is nodal by Lemma \ref{autoelliptic}. Since $g$ acts non-trivially on the $(-1)$-curve $\ell_{13}$, we have $\Aut_{\ct}(S)^\dagger = \{{\rm id}\}$. These are the \textbf{Cases 26 and 28}.

\underline{Case: $s = 4$}
  
By Lemma \ref{non-infinitely near} (4), we may assume $k_4> k_i, i = 1,2,3$. Let $\ell_{123}$ be the line through $q_1,q_2,q_3$ and set $\ell_{i4} := \langle q_i,q_4 \rangle$ for $i = 1,2,3$. Note that $\ell_{123}$ is fixed pointwise by $g$ and it follows from Lemma \ref{fixedpoints} that $(p,n) = 1$.

First, assume that $Q$ is smooth. Then, $\ell$ contains $2$ of the $q_i$ and $Q$ contains the other $2$, while $Q \cap \ell = \{q,q'\}$ for $2$ distinct points $q$ and $q'$ which are interchanged by $g$. This implies that $F_0$ is of type $\tilde{A}_1$, which contradicts our assumption that $F_0$ has at least $4$ components.

Next, assume that $C_0$ consists of $3$ distinct lines. We may assume that $C_0 = \ell_{123} + \ell_{34} + \ell$ for some line $\ell \neq \ell_{14}$. Then, we have $k_1 = 1$, so $k_4 > 3$ and thus $q_4 \in \ell$. We have the following subcases:

\begin{itemize}
\item If $\ell \neq \ell_{24}$, then $(k_1,k_2,k_3,k_4) = (1,1,3,4)$ and $F_0$ is of type $\tilde{A}_7$. Also, $g$ acts non-trivially on $E_1$, so $\Aut_{\ct}(S)^\dagger = \{{\rm id}\}$. Let $F_1$ be the second $g$-invariant fiber and let $C_1 := \pi(F_1)$. All lines that become $(-2)$-curves on $S$ are part of $C_0$, so $C_1$ is irreducible. So, Lemma \ref{autoelliptic} shows that $n = 2$ and $F_1$ is smooth. Hence, this is \textbf{Case 22}.
\item If $\ell = \ell_{24}$, then $(k_1,k_2,k_3,k_4) = (1,2,2,4)$. Since $F_0$ is of type $\tilde{A}_m$ for some $m$, we have $x_i^{(2)} \not \in \ell_{i4}$ for $i = 2,3$. But $x_4^{(2)}$ cannot be contained in both $\ell_{24}$ and $\ell_{34}$, so one of them does not become a $(-2)$-curve on $S$. This contradiction shows that this case does not occur. 
\end{itemize}

Now, assume that $C_0 = \ell_{123} + 2\ell$ for some line $\ell$ through $q_4$. We may assume that $\ell$ does not pass through $q_1,q_2$, so $k_1 = k_2 = 1$. 
We have the following subcases:

\begin{itemize}
\item If $\ell \neq \ell_{34}$, then $(k_1,k_2,k_3,k_4) = (1,1,1,6)$, then $F_0$ is of type $\tilde{E}_6$. Since $g$ acts non-trivially on the $(-1)$-curve $\til_{14}$, we have $\Aut_{\ct}(S)^\dagger = \{{\rm id}\}$. Let $F_1$ be the second $g$-invariant fiber, set $C_1 := \pi(F_1)$, and observe that $C_1$ is irreducible. By Lemma \ref{autoelliptic}, $C_1$ is smooth and $n = 2$, so this leads to \textbf{Cases 24, 25, 27, and 29}.
\item If $\ell = \ell_{34}$, then $k_3 \geq 2$ since $\til_{34}$ and $\til_{123}$ are both part of the same fiber. Then, the line with $x_3^{(2)}$ is $g$-invariant, so it coincides with $\ell_{34}$. Thus, we must have $(k_1,k_2,k_3,k_4) = (1,1,5,2)$, contradicting our assumption that $k_4 > k_3$.
\end{itemize}

Finally, assume that $C_0 = 2\ell_{123} + \ell$ for some line $\ell$ through $q_4$. Since $k_4 \geq 3$, $\ell$ has to be different from the $\ell_{i4}$. Then, we have $(k_1,k_2,k_3,k_4) = (2,2,2,3)$ and $F_0$ is of type $\tilde{D}_4$. There is another fiber $F_1$ of $f$ of type $\tilde{D}_4$ with $\pi(F_1) = \ell_{14} + \ell_{24} + \ell_{34}$. So, this is \textbf{Case 31}.
\end{proof}
 
 \subsection{Cohomologically trivial automorphisms of non-jacobian rational genus one surfaces}

Let us now classify non-jacobian rational elliptic or quasi-elliptic surfaces $f:S\to \bbP^1$ with non-trivial group $\Aut_{\ct}(S)$. Recall that by the results of Section \ref{action}, we have exact sequences
\begin{eqnarray}
0 \to \MW(J(f)) \cap \Aut_{\ct}(S) \to \Aut_{\ct}(S)^\dagger \to \Aut_{\ct}(J(S))^\dagger \\
0 \to \Aut_{\ct}(S)/\Aut_{\ct}(S)^\dagger \to \Aut_{\ct}(J(S))/\Aut_{\ct}(J(S))^\dagger
\end{eqnarray}

Let us first study the cases where $ \MW(J(f)) \cap \Aut_{\ct}(S)$ is non-trivial.

\begin{lemma}\label{lem: cohtrivialtranslations}
Let $f: S \to \bbP^1$ be a genus one fibration on a rational surface of index $m > 1$ and assume that $\Aut_{\ct}(S) \cap \MW(J(f))$ is non-trivial. Then, one of the following cases occurs:
\begin{enumerate}
\item $f$ is elliptic, $2 \mid m$, $(p,m) = 1$, $F_0$ is of type $\tilde{A}_7$ and $\Aut_{\ct}(S) = \bbZ/2\bbZ$.
\item $f$ is elliptic, $3 \mid m$, $(p,m) = 1$, $F_0$ is of type $\tilde{A}_8$ and $\Aut_{\ct}(S) = \bbZ/3\bbZ$.
\item  $m = p = 2$, $F_0$ is of type $\tilde{D}_8$, $\Aut_{\ct}(S) \cap \MW(J(f)) = \Aut_{\ct}(S)^\dagger = \bbZ/2\bbZ$, and $\Aut_{\ct}(S) = \bbG_a$. If $f$ is elliptic, then $\Aut_{\ct}(S)^\dagger \subseteq (\bbZ/2\bbZ)^2$ and if $f$ is quasi-elliptic, then $\Aut_{\ct}(S)^\dagger = \bbZ/2\bbZ$.
\end{enumerate}
\end{lemma}

\begin{proof}
If $F_t$ is  a simple fiber of $f$, then $f$ is isomorphic to $J(f)\to \bbP^1$ in an \'etale neighborhood of $t$ in such a way that the action of $\MW(J(f))$ on $J(f)$ is identified with the action of $\MW(J(f))$ on $f$. Hence, if $g \in \Aut_{\ct}(S) \cap \MW(J(f))$, then it must be a translation automorphism by a section $\Sigma$ of $J(f)$ which meets every fiber corresponding to a simple fiber of $f$ in the identity component. In particular, the section $\Sigma$ meets only one fiber of $J(f)$ in a non-identity component. On the other hand, since $g$ preserves every $m$-section of $f$, the order $n$ of $\Sigma$ divides $m$ and in particular $\Sigma$ is an $n$-torsion section of $J(f)$. 

Assume first that $f$ is elliptic. We apply the theory of Mordell-Weil lattices \cite{ShiodaSchutt}. The height $h(P)$ of a torsion section $\Sigma$ is equal to zero, hence Theorem 6.24 from loc.cit. shows that the sum of the local contributions  is greater than or equal to $2$. Since only one fiber defines a non-zero local contribution, Table 6.1 in loc.cit. shows that this can happen only in the case when the type of the fiber is $\tilde{A}_7, \tilde{A}_8$, or $\tilde{D}_8$. It also shows that the torsion section does not intersect the zero section and has order $2$, $3$, or $2$, respectively. In the first two cases, we have $(p,m) = 1$, since $F_0$ is multiplicative. In the third case, we have $m = p$, since $F_0$ is additive. But, as explained above, $n$ divides $m$, hence $m = p = n = 2$. 

Next, if $f$ is quasi-elliptic, then the classification of Mordell-Weil groups of quasi-elliptic surfaces (see \cite{Ito1}, \cite{Ito2}) shows that $\Sigma$ as above exists if and only if $p = n = 2$ and $f$ admits a fiber of type $\tilde{D}_8$. As in the previous paragraph, we also have $m = 2$.

Finally, by Theorem \ref{delPezzo} (3), the group $\Aut_{\ct}(S) \cap \MW(J(f))$ is a subgroup of $\Aut_{\ct}(S')$ for some jacobian rational genus one surface $f':S' \to \bbP^1$ which admits a fiber $F_0'$ of the same type as $F_0$. Thus, in Case (1) and (2), Table \ref{Table2} shows that $\Aut_{\ct}(S)$ is the stated group. In Case (3), we know from Table \ref{automorphismsofequations}, that $\Aut_{\ct}(S') = \bbG_a$ acts trivially on the identity component of $F_0'$. Thus, by Theorem \ref{delPezzo} (3), we have $\Aut_{\ct}(S) = \Aut_{\ct}(S') = \bbG_a$. Moreover, we have $\MW(J(f)) = \Aut_{\ct}(S) \cap \MW(J(f)) = \bbZ/2\bbZ$ in this case. If $f$ is quasi-elliptic, then $\Aut_{\ct}(J(S))^\dagger$ is trivial, so $\Aut_{\ct}(S) \cap \MW(J(f)) = \Aut_{\ct}(S)^\dagger = \bbZ/2\bbZ$. If $f$ is elliptic, then $\Aut_{\ct}(J(S))^\dagger = \bbZ/2\bbZ$, so $\Aut_{\ct}(S)^\dagger$ has order at most $4$. As $\Aut_{\ct}(S)^\dagger$ is contained in $\bbG_a$, this shows $\Aut_{\ct}(S)^\dagger \subseteq (\bbZ/2\bbZ)^2$.
\end{proof}

Next, we are interested in the group $\Aut_{\ct}(S)/\Aut_{\ct}(S)^\dagger$. In Theorem \ref{thm: actiononjacobian} (5), we have seen that $\Aut_{\ct}(S)$ acts trivially on the identity component of the fiber $J_0$ of $J(f)$ corresponding to the multiple fiber $mF_0$ of $f$ as soon as $F_0$ has more than two components. Using the results of the previous subsection, we can describe precisely all the cases where such cohomologically trivial automorphisms of $J(S)$ exist. In the following lemma, we are referring to the cases in Table \ref{Table2}.

\begin{lemma}\label{lem: trivialonfiberautomorphisms}
Let $f: J \to \bbP^1$ be a jacobian rational genus one surface and let $J_0$ be a fiber. Let $G \subseteq \Aut_{\ct}(J)$ be the subgroup of automorphisms preserving $J_0$ and acting trivially on the identity component $J_0^\sharp$ of $J_0$. Then, the following hold:
\begin{enumerate}
\item If $F_0$ is additive, then $G \cap \Aut_{\ct}(J)^\dagger$ is the $p$-torsion subgroup of $\Aut_{\ct}(J)^\dagger$.
\item If $F_0$ is not additive, then $G \cap \Aut_{\ct}(J)^\dagger$ is trivial.
\item $G \not \subseteq \Aut_{\ct}(J)^\dagger$ holds precisely in the following cases:
\begin{enumerate}
\item $p \neq 5$, $J$ is as in Case $1$, $F_0$ is of type $\tilde{E}_8$ and $G = \bbZ/5\bbZ$.
\item $p \neq 3$, $J$ is as in Case $13$, $F_0$ is of type $\tilde{E}_7$ and $G = \bbZ/3\bbZ$.
\item $p \neq 2$, $J$ is as in Case $23$, $F_0$ is of type $\tilde{E}_6$ and $G = \bbZ/2\bbZ$.
\item $J$ is as in Case $4$ or $7$, $F_0$ is of type $\tilde{E}_8$ and $G = \bbG_a$.
\item $J$ is as in Case $5$ or $8$, $F_0$ is of type $\tilde{E}_8$ and $G = \mathbb{G}_a \rtimes \bbZ/5\bbZ$.
\item $J$ is as in Case $10$ or $11$, $F_0$ is of type $\tilde{D}_8$ and $G = \bbG_a$.
\item $J$ is as in Case $12,15,$ or $18$, $F_0$ is of type $\tilde{A}_8$ or $\tilde{E}_7$ and $G = \bbZ/3\bbZ$.
\item $J$ is as in Case $21,22,24,25,26,27,28,$ or $29,$ $F_0$ is of type $\tilde{A}_7$ or $\tilde{E}_6$ and $G = \bbZ/2\bbZ$.
\item $J$ is as in Case $17$, $F_0$ is of type $\tilde{E}_7$ and $G = \bbZ/2\bbZ$.
\end{enumerate}
\end{enumerate}
\end{lemma}

\begin{proof}
To prove the first two claims we use that the fixed locus of $g\in G$ contains the zero section and $J_0^\sharp$, hence it is not smooth. This implies that  $G \cap \Aut_{\ct}(S)^\dagger$ is $p^n$-torsion. If $J_0$ is additive, then every automorphism of order $p^n$ in 
$\Aut_{\ct}(J)^\dagger$ fixes $J_0$ pointwise (by Lemma \ref{autoelliptic} 
and the known structure of fixed loci on $\bbP^1$). So, $G \cap \Aut_{\ct}(J)^\dagger$ coincides with the $p$-torsion subgroup of $\Aut_{\ct}(J)^\dagger$. If $J_0$ is 
multiplicative or smooth, one easily checks that $G \cap \Aut_{\ct}(J)^\dagger$ is 
trivial, since it does not contain the inversion involution.
The third claim follows from the explicit description of $\Aut_{\ct}(J(S))\setminus \Aut_{\ct}(J(S))^\dagger$ in Table \ref{automorphismsofequations}.
\end{proof}

It turns out that an automorphism in $\Aut_{\ct}(J) \setminus \Aut_{\ct}(J)^\dagger$ which acts non-trivially on $J_0$ rarely comes from an automorphism of $S$.

\begin{lemma} \label{lem: exception}
Let $f: S \to \bbP^1$ be a rational genus one surface of index $m > 1$. Let $mF_0$ be the multiple fiber of $f$ and let $J_0$ be the corresponding fiber of the jacobian fibration. Assume that there exists $g \in \Aut_{\ct}(S) \setminus \Aut_{\ct}(S)^\dagger$ such that $\varphi(g)$ acts non-trivially on the identity component $J_0^\sharp$ of $J_0$. Then, $J(f)$ is as in Case $2$ or $9$, $m = 2$, $F_0$ is smooth, $\Aut_{\ct}(S) = \bbZ/4\bbZ$ and $\Aut_{\ct}(S)^\dagger = \bbZ/2\bbZ$.
\end{lemma}

\begin{proof}
Since $\varphi(g)$ acts non-trivially $J_0^\sharp$, we know that $F_0$ is of type $\tilde{A}_0,\tilde{A}_0^*$ or $\tilde{A}_1$ by Theorem \ref{thm: actiononjacobian} (5). Thus, by Lemma \ref{lem: trivialonfiberautomorphisms} (3), $g$ acts faithfully on $J_0^\sharp$. In particular, we have $n := \ord(g) \leq 4$. 

If some non-trivial power of $g$ acts trivially on $\bbP^1$, then $n = 4$, $m = 2$ and $\varphi(g)^2 \in \Aut_{\ct}(J(S))^\dagger$. Using Table \ref{automorphismsofequations}, we see that this can happen only in Case $2$ or $9$ and if $F_0$ is smooth. This is exactly the situation described in the statement of the lemma. We may thus assume that $g$ acts faithfully on $\bbP^1$.

By Lemma \ref{non-infinitely near} (5), $f$ admits a reducible fiber different from $F_1$ and since $g$ preserves such a fiber, the action of $g$ on $\bbP^1$ has two fixed points, so we have $p \nmid n$. Now, as $g$ acts faithfully on $\bbP^1$, the quotient surface $S/(g)$ admits a genus one fibration $\bar{f}:\bar{S}:= S/(g)\to \bbP^1$. Let $\sigma:\bar{S}'\to \bar{S}$ be its minimal resolution and let $\tau:\bar{S}'\to S'$ be a birational morphism onto a relatively minimal genus one fibration $f':S'\to \bbP^1$. We have the following commutative diagram
$$
\xymatrix{&\bar{S}'\ar[d]^\sigma\ar[dr]^\tau&\\
S\ar[d]_f\ar[r]^\phi&S/(g)\ar[d]_{\bar{f}}\ar@{-->}[r]&S'\ar[dl]^{f'}\\
\bbP^1\ar[r]^{\bar{\phi}}&\bbP^1&}
$$
Let $\Phi =\tau\circ \sigma^{-1}\circ \phi$, so that we get a commutative diagram of rational maps
 $$
\xymatrix{S\ar[d]^f\ar@{-->}[r]^{\Phi}&S'\ar[d]_{f'}\\
\bbP^1\ar[r]^{\bar{\phi}}&\bbP^1.}
$$
Let $E$ be an $m$-section of $S$. Since $E$ and $F_0$ are invariant and meet transversally in a single point on $S$, it is easy to check that $\Phi(E)$ and $F_0' := F_{\bar{\phi}(f(F_0))}$ meet transversally in a single point on $S'$. On the other hand, $\Phi(E)$ is an $m$-section of $f'$. This shows that $f'$ has index $m$ and $F_0'$ is the multiple fiber of $f'$. Now, by the known behaviour of singular fibers of elliptic fibrations under tame base change (see e.g. \cite[Section 5.2]{SchuettShioda}) and since $F_0'$ is reducible, the fiber $F_0'$ is of additive type. This implies $m = p$ and since $m \mid n$ in every case of Theorem \ref{thm: actiononjacobian} (5), this contradicts $p \nmid n$.
\end{proof}

Now we are able to go to the classification of groups $\Aut_{\ct}(S)$ for non-jacobian genus one surfaces.

\begin{theorem}
Let $f: S \to \bbP^1$ be a non-jacobian rational genus one surface.
\begin{enumerate}
\item If $\Aut_{\ct}(S)$ is non-trivial, then the singular fibers and the multiple fiber are as in the following Table \ref{nonjacobiantable}.
\item In every case, Table \ref{nonjacobiantable} also gives an upper bound on $\Aut_{\ct}(S)^\dagger$ and on the group $\Aut_{\ct}(S)$. 
\item In every case, if $f$ is elliptic or satisfies condition ${\rm (OS)}$ from Definition \ref{def: quasiellipticoggshafarevich}, then Table \ref{nonjacobiantable} gives a bound on the number of moduli.
\end{enumerate}

\begin{table}
\scalebox{0.9}{
$\displaystyle
\begin{array}{|l|c|c|c|c|c|c|c|}
\hline
&\text{Singular fibers} & \text{Multiple fiber} & e \text{ or } qe  & \Aut_{\ct}(S)^\dagger & \Aut_{\ct}(S) &  \text{Moduli} & p \\ \hline
1'&\tilde{E}_8,\tilde{A}_0^{**} & 2\tilde{A}_0 & e  & \bbZ/2\bbZ & \bbZ/2\bbZ &0 & \neq 2,3\\
1''&\tilde{E}_8,\tilde{A}_0^{**} & 3\tilde{A}_0 & e  & \bbZ/3\bbZ & \bbZ/3\bbZ &0 & \neq 2,3\\
1''' &\tilde{E}_8,\tilde{A}_0^{**} & p\tilde{E}_8 & e  & \{\id\} & \bbZ/5\bbZ &0 & \neq 0,2,3,5\\

2' &\tilde{E}_8, \tilde{A}_0^*, \tilde{A}_0^* & 2\tilde{A}_0 & e  &  \bbZ/2\bbZ & \bbZ/4\bbZ & 1 & \neq 2,3 \\
2''&\tilde{E}_8, \tilde{A}_0^*, \tilde{A}_0^* & 2\tilde{A}_0^* & e  &  \bbZ/2\bbZ & \bbZ/2\bbZ & 0 & \neq 2,3 \\

3'&\tilde{E}_8, \tilde{A}_0^* &  2\tilde{A}_0 & e  
&
\bbZ/2\bbZ & \bbZ/2\bbZ & 1 & 3 \\
3''&\tilde{E}_8, \tilde{A}_0^* &  2\tilde{A}_0^* & e  
&
\bbZ/2\bbZ & \bbZ/2\bbZ & 0 & 3 \\

4'&\tilde{E}_8 &  2\tilde{A}_0 & e & \bbZ/2\bbZ& \bbZ/2\bbZ  & 0 & 3 \\
4'' &\tilde{E}_8 &  3\tilde{E}_8 & e & \bbZ/3\bbZ & \bbG_a  & 1 & 3 \\

5' &\tilde{E}_8 & 3\tilde{E}_8 & qe & \{\id\} & \bbG_a \rtimes \bbZ/5\bbZ & 0 & 3  \\ 

6'&\tilde{E}_8, \tilde{A}_0^* & 2\tilde{A}_0 & e & 
\bbZ/2\bbZ & \bbZ/2\bbZ & 1 &2  \\
6''&\tilde{E}_8, \tilde{A}_0^* & 2\tilde{E}_8 & e & 
\bbZ/2\bbZ & \bbZ/2\bbZ & 1 &2  \\

7'&\tilde{E}_8 & 2\tilde{E}_8 & e &Q_8 & (\bbZ/2\bbZ)^2 \cdot \bbG_a & 1 & 2 \\

8'&\tilde{E}_8 & 2\tilde{E}_8 & qe &  \{\id\} & \bbG_a \rtimes \bbZ/5\bbZ & 0 & 2 \\ \hline

9'&\tilde{D}_8, \tilde{A}_0^*, \tilde{A}_0^* & 2\tilde{A}_0 & e & \bbZ/2\bbZ & \bbZ/4\bbZ & 0 & \neq 2 \\
9''&\tilde{D}_8, \tilde{A}_0^*, \tilde{A}_0^* & 2\tilde{A}_0^* & e & \bbZ/2\bbZ & \bbZ/2\bbZ & 0 & \neq 2 \\

10'&\tilde{D}_8 & 2\tilde{A}_0 &e & \bbZ/2\bbZ & \bbZ/2\bbZ & 2 & 2\\
10''&\tilde{D}_8 & 2\tilde{D}_8 &e & (\bbZ/2\bbZ)^2 & \bbG_a & 2 & 2\\

11'&\tilde{D}_8 & 2\tilde{D}_8 & qe &  \bbZ/2\bbZ & \bbG_a & 1 & 2 \\ \hline

12'&\tilde{A}_8, \tilde{A}_0^*,\tilde{A}_0^*,\tilde{A}_0^* & m\tilde{A}_8, (m,p) = 1 & e  & \bbZ/3\bbZ  & \bbZ/3\bbZ  & 0 & \neq 3 \\ \hline

13'&\tilde{E}_7,\tilde{A}_1^{*} & 2\tilde{A}_0   & e & \bbZ/4\bbZ& \bbZ/4\bbZ& 0 & \neq 2 \\
13''&\tilde{E}_7,\tilde{A}_1^{*} & p\tilde{E}_7   & e/qe & \{\id\}& \bbZ/3\bbZ& 0 & \neq 0,3 \\

14'&\tilde{E}_7, \tilde{A}_1, \tilde{A}_0^* & 2\tilde{A}_0 & e &  \bbZ/2\bbZ &  \bbZ/2\bbZ & 1 & \neq 2   \\
14''&\tilde{E}_7, \tilde{A}_1, \tilde{A}_0^* & 2\tilde{A}_0^* & e &  \bbZ/2\bbZ &  \bbZ/2\bbZ & 0 & \neq 2   \\
14'''&\tilde{E}_7, \tilde{A}_1, \tilde{A}_0^* & 2\tilde{A}_1 & e &  \bbZ/2\bbZ &  \bbZ/2\bbZ & 0 & \neq 2   \\

15'&\tilde{E}_7,\tilde{A}_0^{*},\tilde{A}_0^{*},\tilde{A}_0^{*} & p\tilde{E}_7 & e& \{\id\}& \bbZ/3\bbZ& 1 &\ne  0,2,3 \\

16'&\tilde{E}_7, \tilde{A}_1 & 2\tilde{A}_0 & e &   \bbZ/2\bbZ & \bbZ/2\bbZ & 1 & 2  \\
16''&\tilde{E}_7, \tilde{A}_1 & 2\tilde{E}_7 & e &   \bbZ/2\bbZ & \bbZ/2\bbZ & 1 & 2  \\

17'&\tilde{E}_7,\tilde{A}_0^{*},\tilde{A}_0^{*}& 2\tilde{E}_7 & e&   \{\id\}& \bbZ/2\bbZ & 2 &2  \\

18'&\tilde{E}_7 & 2\tilde{E}_7 & e &  \{\id\} &  \bbZ/3\bbZ  & 1 & 2 \\ \hline

21'&\tilde{A}_7, \tilde{A}_1, \tilde{A}_0^*, \tilde{A}_0^* & m\tilde{A}_7, (m,p) = 1 & e & \bbZ/2\bbZ & \bbZ/2\bbZ  & 0 & \neq 2  \\

22'&\tilde{A}_7, \tilde{A}_0^*, \tilde{A}_0^*, \tilde{A}_0^*, \tilde{A}_0^* & m\tilde{A}_7, (m,p) = 1 & e & \bbZ/2\bbZ & \bbZ/2\bbZ  & 1 & \neq 2 \\ \hline

23'&\tilde{E}_6, \tilde{A}_2^* & 3\tilde{A}_0& e & \bbZ/3\bbZ & \bbZ/3\bbZ & 0 & \neq 3 \\
23''&\tilde{E}_6, \tilde{A}_2^* & p\tilde{E}_6& e/qe & \{\id\} & \bbZ/2\bbZ & 0 & \neq 0,2 \\

24'&\tilde{E}_6, \tilde{A}_0^*, \tilde{A}_0^*, \tilde{A}_0^*, \tilde{A}_0^* & p \tilde{E}_6 & e   & \{\id\} & \bbZ/2\bbZ  & 2 & \neq 0,2,3  \\

25'&\tilde{E}_6,\tilde{A}_0^{**},\tilde{A}_0^{**} & p \tilde{E}_6 & e  & \{\id\} & \bbZ/2\bbZ  & 1 & \neq 0,2,3  \\

26'&\tilde{E}_6,\tilde{A}_1,\tilde{A}_0^*,\tilde{A}_0^* & p \tilde{E}_6 & e  & \{\id\} & \bbZ/2\bbZ  & 1 & \neq 0,2,3  \\

27'&\tilde{E}_6, \tilde{A}_0^*, \tilde{A}_0^*  & 3 \tilde{E}_6 & e &  \{\id\} & \bbZ/2\bbZ  & 2 & 3 \\

28'&\tilde{E}_6, \tilde{A}_1  & 3 \tilde{E}_6 & e   & \{\id\} & \bbZ/2\bbZ  & 1 & 3 \\

29'&\tilde{E}_6 & 3 \tilde{E}_6 & e  &  \{\id\} & \bbZ/2\bbZ  & 1 & 3 \\ \hline

30'&\tilde{D}_6,  \tilde{A}_1, \tilde{A}_1 & 2\tilde{A}_0 & e &   \bbZ/2\bbZ & \bbZ/2\bbZ & 1 & \neq 2 \\
30''&\tilde{D}_6,  \tilde{A}_1, \tilde{A}_1 & 2\tilde{A}_1 & e &   \bbZ/2\bbZ & \bbZ/2\bbZ & 0 & \neq 2 \\ \hline
32'&\tilde{D}_4,  \tilde{D}_4 & 2\tilde{A}_0 & e & \bbZ/2\bbZ & \bbZ/2\bbZ & 1 & \neq 2 \\ \hline
\end{array}
$}
\bigskip
\caption{Non-jacobian rational genus one surfaces with cohomologically trivial automorphisms}
\label{nonjacobiantable}
\end{table}

\end{theorem}

\begin{proof}

Let  $mF_0$ be the multiple fiber of $f$ and let $J_0$ be the corresponding fiber of $J(f)$ and let $J_0^0$ be its identity component. 

If $\Aut_{\ct}(S) \cap \MW(J(f))$ is non-trivial we apply Lemma \ref{lem: cohtrivialtranslations} and we get Cases $10''$, $11'$, $12'$, $21'$, and $22'$.

In the other cases, $\Aut_{\ct}(S)$ acts faithfully on $J(S)$. We will distinguish several cases according to the type of $F_0$.

If $F_0$ is of additive type and reducible, then $m = p$. By Theorem \ref{thm: actiononjacobian} (5), $\Aut_{\ct}(S)$ acts trivially on $J_0^0$. Using Lemma \ref{lem: trivialonfiberautomorphisms} (1) and (3), one easily checks that we get the following Cases: $1''',4'',5',6'',7',8',10'',11',13'',15',16'',17',18',23'',24',25',26',27',28'$, and $29'$. The same lemma also gives bounds on $\Aut_{\ct}(S)^\dagger$ and $\Aut_{\ct}(S)$.

If $F_0$ is of additive type and irreducible, then we also have $m = p$. By Theorem \ref{thm: actiononjacobian} (5), $\Aut_{\ct}(S)$ acts trivially on $J_0^0$, so by Lemma \ref{lem: trivialonfiberautomorphisms}, we have $\Aut_{\ct}(S)^\dagger = \Aut_{\ct}(S)$ and both groups are $p$-torsion groups contained in $\Aut_{\ct}(J(S))^\dagger$. In all cases in Table \ref{Table2} that admit a fiber of type $\tilde{A}_0^{**}$, the group $\Aut_{\ct}(J(S))^\dagger$ is tame, so this case does not occur.

If $F_0$ is of multiplicative type with at least $3$ components, then $(m,p) = 1$. By Theorem \ref{thm: actiononjacobian} (5), $\Aut_{\ct}(S)$ acts trivially on $J_0^0$, so by Lemma \ref{lem: trivialonfiberautomorphisms} (3), either $F_0$ is of type $\tilde{A}_8$ and we get Case $12'$, or $F_0$ is of type $\tilde{A}_7$ and we get Cases $21'$ and $22'$.

If $F_0$ is of type $\tilde{A}_1$, then we also have $(m,p) = 1$. By Lemma \ref{lem: trivialonfiberautomorphisms}, $\Aut_{\ct}(S)$ acts faithfully on $J_0^0$, so Theorem \ref{thm: actiononjacobian} (5) shows that $\Aut_{\ct}(S) = \bbZ/2\bbZ$ and $p \neq 2$.
By Lemma \ref{lem: exception}, we have $\Aut_{\ct}(S)^\dagger = \Aut_{\ct}(S)$.
The only fibrations in Table \ref{Table2} with $\Aut_{\ct}(J(S))^\dagger = \bbZ/2\bbZ$ and a fiber of type $\tilde{A}_1$ in characteristic $p \neq 2$ are Cases $14$ and $30$. Thus, we get Cases $14'''$ and $30''$.

If $F_0$ is of type $\tilde{A}_0^{\ast}$, we have $(m,p) = 1$. By the same argument as in the previous paragraph, we have $\Aut_{\ct}(S)^\dagger = \Aut_{\ct}(S) = \bbZ/2\bbZ$. This leads to Cases $2'',3'',9'',$ and $14''$.

Finally, if $F_0$ is smooth, then Lemma \ref{lem: trivialonfiberautomorphisms} shows that $\Aut_{\ct}(S)$ acts faithfully on $J_0^0$ and we can apply Theorem \ref{thm: actiononjacobian} (5) to get restrictions on $m$, $\Aut_{\ct}(S)^\dagger$, and $\Aut_{\ct}(S)$. By Lemma \ref{lem: exception}, we have $\Aut_{\ct}(S)^\dagger = \Aut_{\ct}(S)$ unless we are in Case $2'$ or $9'$. If $m = 3$, we obtain Cases $1''$ and $23'$. If $m = 2$, we obtain Cases $1',3',4',6',10',13',14',16',30',$ and $31'$.   

To obtain a bound on the number of moduli we use the Ogg-Shafarevich theory of elliptic fibrations (see the discussion before Definition \ref{def: quasiellipticoggshafarevich}). According to this theory, a tame rational torsor of index $m$ is determined uniquely by a choice of a multiple fiber $F_{t_0}$ and a choice of a point $P$ of order $m$ in the connected component of the identity of $J_{t_0}^\sharp$. So, the number of moduli for the surfaces in Table \ref{nonjacobiantable} is the number of moduli of its jacobian $J(S)$ plus the number of moduli for choices of $F_{t_0}$ and $P$ up to automorphisms of $J(S)$. We leave it to the reader to calculate this number in every case. 
\end{proof}

\begin{example} \label{ex: nonjacobianexample1}
Assume $p \neq 2,3$. Let $J(f): J(S) \to \bbP^1$ be the jacobian rational elliptic surface with singular fibers of type $\tilde{E}_8,\tilde{A}_0^*,\tilde{A}_0^*$ and zero section $E$. Let $g \in \Aut_{\ct}(J(S))$ be a cohomologically trivial automorphism of order $4$. Let $J_0$ be the smooth fiber of $J(f)$ preserved by $g$ and let $P \in J_0$ be the fixed point of $g$ which does not lie on $E$. Note that $P$ is a $2$-torsion point on $J_0$. As mentioned in Remark \ref{rem: Halphenremark}, blowing down $E$ and blowing up $P$, we obtain a rational genus one surface $f': S' \to \bbP^1$ of index $2$ together with a $2$-section $E'$, the inverse image of $P$, and a cohomologically trivial automorphism $g'$ of order $4$. Using the description of $g$ in Table \ref{automorphismsofequations}, one checks that $g'$ restricts to an automorphism of order $4$ on $E'$ and thus $g'$ acts non-trivially on $\bbP^1$. Therefore, by Table \ref{nonjacobiantable}, the fibration $f'$ is of type $2'$ or $9'$. In particular, some of the exceptional cases described in Lemma \ref{lem: exception} exist.
\end{example}

\begin{example} \label{ex: nonjacobianexample2}
If $p \neq 0,2,3$ and $J(f): J(S) \to \bbP^1$ is a jacobian rational elliptic surface with singular fibers of type $\tilde{E}_6,\tilde{A}_0^*,\tilde{A}_0^*,\tilde{A}_0^*,\tilde{A}_0^*$ and with a cohomologically trivial involution $g$, then $g$ acts trivially on the identity component of the fiber of type $\tilde{E}_6$ by Table \ref{automorphismsofequations}. Thus, similarly to Example \ref{ex: nonjacobianexample1}, one can use any $p$-torsion point on the fiber of type $\tilde{E}_6$ to construct a non-jacobian rational elliptic surface with a multiple fiber of type $\tilde{E}_6$ and index $p$. Moreover, the resulting surface admits a cohomologically trivial involution, hence it is of type $24', 25',$ or $26'$.
\end{example}

\begin{remark}
We do not claim that all cases in Table \ref{nonjacobiantable} exist. We leave it to the interested reader to try to realize all cases using Halphen transforms as in Examples \ref{ex: nonjacobianexample1} and \ref{ex: nonjacobianexample2}.
\end{remark}

 \section{Numerically trivial automorphisms of classical Enriques surfaces in characteristic $2$}
 We refer the reader to \cite{CDL} for the basic facts on Enriques surfaces which we will use here without proof. Every Enriques surface $S$ admits at least one genus one fibration $f: S \to \bbP^1$ and the jacobian fibration $J(f)$ of $f$ is a rational genus one surface. Moreover, the fibration $f$ is cohomologically flat if and only if the Enriques surface $S$ is classical, that is, if $\omega_S \neq \calO_S$ (which always holds if the characteristic is not $2$). In this case, $f$ has precisely two double fibers. In this section, we give an application of Theorem \ref{thm: actiononjacobian} and the calculations of Section \ref{sec: cohtrivialjacobian} to classical Enriques surfaces. This recovers some of the results of \cite{DolgachevMartin} and extends them to characteristic $2$. We let $\Aut_{\nt}(S)$ be the subgroup of $\Aut(S)$ of \emph{numerically trivial automorphisms}, which are those automorphisms of $S$ acting trivially on $\Num(S)$. In the case $p\ne 2$, all Enriques surfaces with non-trivial $\Aut_{\nt}(S)$ have been classified in \cite[Theorem 8.2.22]{DK}. So, we assume that $p = 2$.  
 
 As we shall see, Theorem \ref{thm: actiononjacobian} allows us to achieve a classification of classical Enriques surfaces with numerically trivial automorphisms in characteristic $2$. Applying this theorem and $\Aut_{\nt}(J(S)) = \Aut_{\ct}(J(S))$, we get the following exact sequences
\begin{eqnarray}
0 \to \MW(J(f)) \cap \Aut_{\nt}(S) \to \Aut_{\nt}(S)^\dagger \to \Aut_{\ct}(J(S))^\dagger \label{seq: enriques1} \\
0 \to \Aut_{\nt}(S)/\Aut_{\nt}(S)^\dagger \to \Aut_{\ct}(J(S))/\Aut_{\ct}(J(S))^\dagger. \label{seq: enriques2}
\end{eqnarray}
 
 Before we start, we recall that an Enriques surface is called $\tilde{E}_8$-extra-special if it contains the following configuration of $(-2)$-curves:
$$
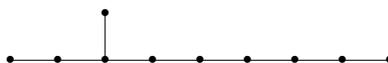
\resizebox{!}{1cm}{
\xy
(-10,25)*{};
@={(-10,10),(0,10),(10,10),(20,10),(30,10),(40,10),(50,10),(10,20),(60,10),(70,10)}@@{*{\bullet}};
(-10,10)*{};(70,10)*{}**\dir{-};
(10,10)*{};(10,20)*{}**\dir{-};
\endxy
} 
$$
Note that the only genus one fibration on a $\tilde{E}_8$-extra-special surface is a quasi-elliptic fibration with a double fiber of type $\tilde{E}_8$. We will first prove the following refinement of \cite[Theorem 8.5]{DolgachevMartin}.

\begin{theorem}\label{thm: dolgachevmartin}
Let $S$ be a classical Enriques surface in characteristic $2$. Then, $\Aut_{\nt}(S) \subseteq (\bbZ/2\bbZ)^b$ with $b \leq 2$.
\end{theorem}

\proof
This was proved in \cite[Theorem 8.5]{DolgachevMartin} if $S$ is not $\tilde{E}_8$-extra-special, so we may assume that we are in this special case and we let $f: S \to \bbP^1$ be the unique genus one fibration of this surface. Note that $J(f)$ corresponds to Case $8$ in Table \ref{Table2}. Since $\MW(J(f))$ is trivial, the group $\Aut_{\nt}(S)$ acts faithfully on $J(S)$ by \eqref{seq: enriques1}. Let $F_0$ and $F_1$ be the double fibers of $f$, say with $F_1$ irreducible. Since $\Aut_{\nt}(S)$ preserves the $F_i$, Table \ref{Table2} shows that $p \nmid |\Aut_{\nt}(S)|$. Then, Table \ref{automorphismsofequations} shows that $\Aut_{\nt}(S)$ acts faithfully on $F_1$. By Theorem \ref{thm: actiononjacobian}, this is only possible if $\Aut_{\nt}(S)$ is trivial.
\qed

Now, we follow the strategy we used in the previous subsection to calculate cohomologically trivial automorphisms of non-jacobian rational genus one surfaces.

\begin{lemma}\label{translations} Let $f:S\to \bbP^1$ be a genus one fibration on a classical Enriques surface in characteristic $p = 2$ with double fibers $F_0,F_1$. Assume that $G:=\Aut_{\nt}(S) \cap MW(J(f)) \ne \{1\}$. Then, one of the following cases occurs.
\begin{enumerate}
\item  $G = (\bbZ/2\bbZ)^2$, $f$ is quasi-elliptic, and $(F_0,F_1) = (\tilde{D}_4,\tilde{D}_4)$.
\item $G = \bbZ/2\bbZ$, and $f$ is elliptic and $(F_0,F_1) = (\tilde{D}_8, \tilde{A}_0)$, or $f$ is quasi-elliptic and $(F_0,F_1) \in \{(\tilde{D}_4,\tilde{D}_4), (\tilde{D}_6,\tilde{A}_1^*), (\tilde{E}_7,\tilde{A}_1^*),(\tilde{D}_8,\tilde{A}_0^{**})\}$.
\end{enumerate}
\end{lemma}

\begin{proof}
Since both double fibers of $f$ are tame, they must be either of additive type or ordinary elliptic curves. As we mentioned earlier, the fibers of the jacobian fibration $J(f):J(S)\to \bbP^1$ are of the same type.  
By the same argument as in Lemma \ref{translations}, the action of $\Aut_{\nt}(S) \cap \MW(J(f))$ is induced by translations by sections of $J(f)$ which meet every fiber of $J(f)$ except possibly the two fibers corresponding to $F_0$ and $F_1$ in the identity component. Since $\Aut_{\nt}(S)$ is 2-torsion by Theorem \ref{thm: dolgachevmartin}, every non-trivial translation automorphism from $\Aut_{\nt}(S)$ is of order 2. If $f$ is an elliptic fibration, as before, we apply the theory of Mordell-Weil lattices. We have now at most two local contributions with sum a positive integer $\ge 2$. Using Table 6.1 from \cite{ShiodaSchutt} we get that the only possibility for the types  of $(F_0,F_1)$ is $(\tilde{D}_8,\tilde{A}_0)$. 

Suppose $f$ is quasi-elliptic. Although quasi-elliptic  fibrations were not considered  in \cite{ShiodaSchutt}, the same arguments apply without change. 
Alternatively, one can use the explicit description of the Mordell-Weil groups of rational quasi-elliptic surfaces in characteristic $2$ given in \cite{Ito2} to obtain the possible types of $F_0$ and $F_1$.
%
%
\end{proof}

\begin{proposition} \label{prop: enriques}
Let $f: S \to \bbP^1$ be a genus one fibration on a classical Enriques surface $S$ in characteristic $2$. If $\Aut_{\nt}(S) \neq \{{\rm id}\}$, then one of the following cases occurs, where $F_0$ and $F_1$ are the double fibers of $f$:
\begin{enumerate}
\item $f$ is as in Lemma \ref{translations}.
\item $f$ is extremal, $f$ admits a fiber of type $\tilde{E}_8,\tilde{D}_8$, or $\tilde{E}_7$ and $J(f)$ occurs in Table \ref{Table2},
\item $f$ admits a fiber of type $\tilde{E}_7$, $J(f)$ corresponds to Case $17$ in Table \Ref{Table2}, $F_0$ and $F_1$ are smooth, and $\Aut_{\nt}(S) = \bbZ/2\bbZ$.
\end{enumerate}
\end{proposition}

\begin{proof}
We may assume that $f$ does not occur in Lemma \ref{translations}, so that $\Aut_{\nt}(S) \cap \MW(J(f))$ is trivial and thus, by Sequence \eqref{seq: enriques2}, $\Aut_{\ct}(J(S))$ is not. By Theorem \ref{thm: dolgachevmartin}, we know that $\Aut_{\nt}(S)$ is $2$-torsion. 
In particular, $J(f)$ occurs in Table \ref{Table2} and by checking which of the cases there admit cohomologically trivial involutions in characteristic $2$, we obtain the stated list of fibers. 

The only non-extremal fibration in Table \ref{Table2} that admits a cohomologically trivial involution is Case \textbf{17}.
In this case, we have  $\Aut_{\nt}(S) = \Aut_{\ct}(J(S)) = \bbZ/2\bbZ$, both acting non-trivially on the base. Hence, the $F_i$ are interchanged by $\Aut_{\nt}(S)$. Since they cannot be nodal, they have to be smooth and ordinary.
\end{proof}

\begin{theorem} \label{thm: Enriquesconfigs}
Let $S$ be a classical Enriques surface in characteristic $2$. If $\Aut_{\nt}(S)$ is non-trivial, then $S$ contains one of the following configurations of $(-2)$-curves:
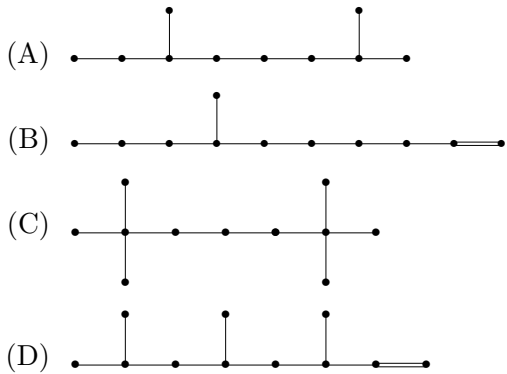
\begin{enumerate}
\item[(A)]
\resizebox{!}{1cm}{
\xy
(-10,25)*{};
@={(-10,10),(0,10),(10,10),(20,10),(30,10),(40,10),(50,10),(10,20),(60,10),(50,20)}@@{*{\bullet}};
(-10,10)*{};(60,10)*{}**\dir{-};
(10,10)*{};(10,20)*{}**\dir{-};
(50,10)*{};(50,20)*{}**\dir{-};
\endxy
} 
\item[(B)] 
\resizebox{!}{1cm}{
\xy
(0,25)*{};
@={(80,10),(90,10),(0,10),(10,10),(20,10),(30,10),(40,10),(50,10),(60,10),(70,10),(30,20)}@@{*{\bullet}};
(0,10)*{};(80,10)*{}**\dir{-};
(90,10)*{};(80,10)*{}**\dir{=};
(30,10)*{};(30,20)*{}**\dir{-};
\endxy
}
\item[(C)]
\resizebox{!}{1cm}{
\xy
(0,15)*{};
@={(0,0),(10,0),(20,0),(30,0),(40,0),(10,-10),(10,10),(40,0),(50,10),(50,0),(50,-10),(60,0)}@@{*{\bullet}};
(0,0)*{};(60,0)*{}**\dir{-};
(10,10)*{};(10,-10)*{}**\dir{-};
(50,10)*{};(50,-10)*{}**\dir{-};
\endxy
 }
\item[(D)]
\resizebox{!}{1cm}{
\xy
(0,15)*{};
@={(30,10),(0,0),(10,0),(20,0),(30,0),(40,0),(50,0),(10,10),(60,0),(50,10),(70,0)}@@{*{\bullet}};
(0,0)*{};(60,0)*{}**\dir{-};
(10,0)*{};(10,10)*{}**\dir{-};
(50,0)*{};(50,10)*{}**\dir{-};
(30,0)*{};(30,10)*{}**\dir{-};
(60,0)*{};(70,0)*{}**\dir{=};
\endxy
} 
\end{enumerate}
In Cases (A), (B), and (C), these are in fact all $(-2)$-curves on $S$.
\end{theorem}

\begin{proof}
Throughout the proof, we assume that $\Aut_{\nt}(S)$ is non-trivial.

Assume that $S$ does not admit a genus one fibration as in Proposition \ref{prop: enriques} (3). Then, all genus one fibrations on $S$ are extremal. By \cite{KatsuraKondoMartin}, this implies that $S$ has finite automorphism group. Comparing the genus one fibrations on Enriques surfaces with finite automorphism group given in \cite[Section 14.1]{KatsuraKondoMartin} with the ones that are allowed by Proposition, we see that we get configurations $(A),(B),(C)$, or one of the following two configurations:
$$
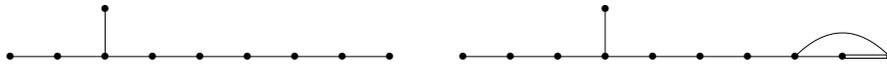
\resizebox{!}{1cm}{
\xy
(-10,25)*{};
@={(-10,10),(0,10),(10,10),(20,10),(30,10),(40,10),(50,10),(10,20),(60,10),(70,10)}@@{*{\bullet}};
(-10,10)*{};(70,10)*{}**\dir{-};
(10,10)*{};(10,20)*{}**\dir{-};
\endxy
}
\quad
\quad
\resizebox{!}{1cm}{
\xy
(-10,25)*{};
@={(-20,10),(-10,10),(0,10),(10,10),(20,10),(30,10),(40,10),(50,10),(10,20),(60,10),(70,10)}@@{*{\bullet}};
(-20,10)*{};(60,10)*{}**\dir{-};
(70,10)*{};(60,10)*{}**\dir{=};
(70,10)*{};(50,10)*{}**\crv{(60,20)};
(10,10)*{};(10,20)*{}**\dir{-};
\endxy
}
$$
The first configuration belongs to the $\tilde{E}_8$-extra-special surfaces, which we have already proved to have trivial group $\Aut_{\nt}(S)$ in the proof of Theorem \ref{thm: dolgachevmartin}. A surface containing the second configuration admits a quasi-elliptic fibration whose jacobian $J(S)$ is as in Case $13$ of Table \ref{Table2}. Since $\Aut_{\nt}(S)$ is $2$-torsion by Theorem \ref{thm: dolgachevmartin} and $\Aut_{\ct}(J(S))$ has odd order, we have $\Aut_{\nt}(S) = \Aut_{\nt}(S) \cap \MW(J(f))$. By Lemma \ref{translations}, this group is trivial in this case. Hence, we are left with configurations $(A),(B),$ and $(C)$.

Next, since every Enriques surface admits some genus one fibration, we may assume that $S$ admits a genus one fibration $f$ as in Proposition \ref{prop: enriques} (3).
By \cite[Theorem 6.1.10]{DK} there is a second genus one fibration $f'$ on $S$ whose general fiber meets $F$ with multiplicity $4$. Let $F'$ be a reducible fiber of $f'$ with at least $5$ irreducible components which exists by Proposition \ref{prop: enriques} and let $E$ be an irreducible component of $F'$ that meets $F$. Since $F$ is a simple fiber, we have $E.F \in \{2,4\}$.

First, assume that $E.F = 4$. Then, $F'$ has to be a simple fiber of $f'$, $E$ is a simple component of $F'$, and we may assume that $F' - E \subseteq F$ because of $(F' - E).F = 0$. Using the list of Proposition \ref{prop: enriques}, we see that $F'$ is of type $\tilde{E}_7$ and we get the following configuration:
$$
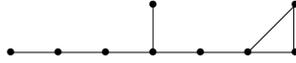
\resizebox{!}{1cm}{
\xy
(-10,25)*{};
@={(-20,10),(-10,10),(0,10),(10,10),(20,10),(30,10),(40,10),(40,20),(10,20)}@@{*{\bullet}};
(-20,10)*{};(40,10)*{}**\dir{-};
(40,20)*{};(40,10)*{}**\dir{=};
(10,10)*{};(10,20)*{}**\dir{-};
(30,10)*{};(40,20)*{}**\dir{-};
\endxy
}
$$
The two right-most curves meet in two distinct points that are interchanged by $\Aut_{\nt}(S)$, for otherwise $\Aut_{\nt}(S)$ would fix these two curves pointwise, which is impossible since $\Aut_{\nt}(S)$ acts non-trivially on the base of $f$ (see the proof of Proposition \ref{prop: enriques}). Hence, these two curves define a (necessarily simple) fiber of type $\tilde{A}_1$ of some genus one fibration $f''$ of $S$. By Proposition \ref{prop: enriques}, $f''$ admits a fiber $F''$ of type $\tilde{E}_7$
 and it is easy to check that some component $E''$ of $F''$ satisfies $E''.F = 2$. 
 
Hence, we may assume that $E.F = 2$. Then, we are in one of the following situations:

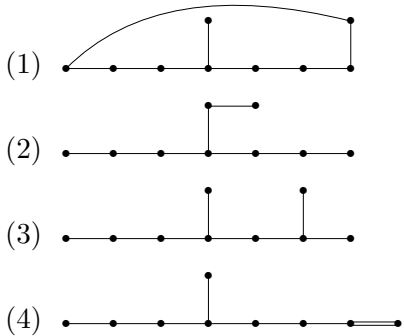
\begin{enumerate}
\item
$
\resizebox{!}{1cm}{
\xy
(-10,25)*{};
@={(-20,10),(-10,10),(0,10),(10,10),(20,10),(30,10),(40,10),(40,20),(10,20)}@@{*{\bullet}};
(-20,10)*{};(40,10)*{}**\dir{-};
(40,20)*{};(40,10)*{}**\dir{-};
(40,20)*{};(-20,10)*{}**\crv{(0,30)};
(10,10)*{};(10,20)*{}**\dir{-};
\endxy
}
$
\item
$
\resizebox{!}{1cm}{
\xy
(-10,25)*{};
@={(-20,10),(-10,10),(0,10),(10,10),(20,10),(30,10),(40,10),(20,20),(10,20)}@@{*{\bullet}};
(-20,10)*{};(40,10)*{}**\dir{-};
(20,20)*{};(10,20)*{}**\dir{-};
(10,10)*{};(10,20)*{}**\dir{-};
\endxy
}
$
\item
$
\resizebox{!}{1cm}{
\xy
(-10,25)*{};
@={(-20,10),(-10,10),(0,10),(10,10),(20,10),(30,10),(40,10),(30,20),(10,20)}@@{*{\bullet}};
(-20,10)*{};(40,10)*{}**\dir{-};
(30,20)*{};(30,10)*{}**\dir{-};
(10,10)*{};(10,20)*{}**\dir{-};
\endxy
}
$
\item
$
\resizebox{!}{1cm}{
\xy
(-10,25)*{};
@={(-20,10),(-10,10),(0,10),(10,10),(20,10),(30,10),(40,10),(50,10),(10,20)}@@{*{\bullet}};
(-20,10)*{};(40,10)*{}**\dir{-};
(50,10)*{};(40,10)*{}**\dir{=};
(10,10)*{};(10,20)*{}**\dir{-};
\endxy
}
$
\end{enumerate}
In Case (1), we find a configuration of type $\tilde{A}_7$, which is impossible by Proposition \ref{prop: enriques}. Similarly, in Case (2), we find a configuration of type $\tilde{E}_6$, which is also impossible by Proposition \ref{prop: enriques}. In Case (3), there is a configuration $F''$ of type $\tilde{D}_6$, which is a double fiber of a quasi-elliptic fibration $f''$ on $S$, since there is a $(-2)$-curve $E'$ on $S$ with   $E'.F'' = 1$.  By Proposition \ref{prop: enriques} and Lemma \ref{translations}, the fibration $f''$ also admits a double fiber of type $\tilde{A}_1^*$. This leads to configuration (D). 

Finally, in Case (4), we find a genus one fibration $f''$ with a double fiber $F_0''$ of type $\tilde{A}_1^*$ such that $f''$ also admits a fiber $F_1''$ containing a configuration of type $D_6$. Here, we used again that there is a $(-2)$-curve $E'$ on $S$ with $E'.F_0'' = 1$.
Note that $F_1''$ cannot be of type $\tilde{E}_7$, for then there would be a $(-2)$-curve as in Case (2), which we have already excluded. Therefore, by Proposition \ref{prop: enriques}, the fiber $F_1''$ is a double fiber of type $\tilde{D}_6$. So, this also leads to configuration (D).
\end{proof}

\begin{remark}
A closer inspection of Enriques surfaces of type (D) shows that every genus one fibration which admits a $(-2)$-curve as a $2$-section on such a surface is of one of the following types:
\begin{enumerate}
\item a quasi-elliptic fibration with a simple fiber of type $\tilde{D}_8$,
\item an elliptic fibration with simple singular fibers of type $\tilde{E}_7,\tilde{A}_0^*,\tilde{A}_0^*$, or
\item a quasi-elliptic fibration with double fibers of type $\tilde{D}_6,\tilde{A}_1^*$ and a simple reducible fiber of type $\tilde{A}_1^*$.
\end{enumerate}
\end{remark}

\begin{theorem} \label{thm: Enriquesgroups}
Let $S$ be a classical Enriques surface in characteristic $2$. If $S$ contains one of the configurations of $(-2)$-curves in Theorem \ref{thm: Enriquesconfigs}, then $\Aut_{\nt}(S)$ and $\Aut_{\ct}(S)$ are as follows:
\begin{enumerate}
\item[(A)] $\Aut_{\nt}(S) = \Aut_{\ct}(S) = \bbZ/2\bbZ$.
\item[(B)] $\Aut_{\nt}(S) = \bbZ/2\bbZ$, $\Aut_{\ct}(S) = \{{\rm id}\}$.
\item[(C)] $\Aut_{\nt}(S) = (\bbZ/2\bbZ)^2$, $\Aut_{\ct}(S) = \{{\rm id}\}$.
\item[(D)] $\Aut_{\nt}(S) = \bbZ/2\bbZ$, $\Aut_{\ct}(S) = \{{\rm id}\}$.
\end{enumerate}
\end{theorem}

\begin{proof}
The claims about $\Aut_{\ct}(S)$ were proven in \cite[Theorem 7.1]{DolgachevMartin}, so we can focus on $\Aut_{\nt}(S)$.

First, assume that $S$ is of type (A). By \cite[Section 14.1]{KatsuraKondoMartin}, there are three genus one fibrations on $S$: A quasi-elliptic fibration $f$ with a double fiber $F_0$ of type $\tilde{D}_8$, an elliptic fibration $f_1$ with a double fiber of type $\tilde{E}_8$, and an elliptic fibration $f_2$ with a simple fiber of type $\tilde{E}_8$. 
By \cite{Ito2}, we have $\MW(J(f)) = \bbZ/2\bbZ$ and since this group cannot interchange the two $\tilde{E}_8$-configurations (since they are fibers of non-isomorphic fibrations), the group $\MW(J(f))$ acts trivially on the configuration of $(-2)$-curves. Since these curves generate $\Num(S) \otimes \bbQ$, we deduce $\MW(J(f)) \subseteq \Aut_{\nt}(S)$.
By Lemma \ref{translations}, we have $\Aut_{\nt}(S) \cap \MW(J(f)) \subseteq \bbZ/2\bbZ$, so it suffices to show that $\Aut_{\nt}(S)$ acts trivially on $J(S)$.
By Table \ref{Table2}, $\Aut_{\nt}(S)$ acts on $J(S)$ via automorphisms with only one fixed point on the base curve $\bbP^1$. Since $f$ has two double fibers, one of which is reducible, this implies that $\Aut_{\nt}(S)$ acts trivially on $J(S)$

Next, assume that $S$ is of type (B). Consider the fibration $f$ with double fibers $F_0$ and $F_1$ of type $\tilde{E}_7$ and $\tilde{A}_1^{*}$. Note that the diagram of $(-2)$-curves has no symmetries and these $(-2)$-curves generate $\Num(S)$, so $\MW(J(f)) = \bbZ/2\bbZ$ acts trivially on $\Num(S) \otimes \bbQ$. Thus, as in Case (A), it suffices to show that $\Aut_{\nt}(S)$ acts trivially on $J(S)$. Seeking a contradiction, assume that it does not. Then, by Table \ref{automorphismsofequations}, $\Aut_{\nt}(S)$ acts non-trivially on the identity component of the fiber of $J(f)$ corresponding to $F_1$. This is impossible by Theorem \ref{thm: actiononjacobian} (5).

Now, assume that $S$ is of type (C).  Consider the fibration $f$ with double fibers $F_0$ and $F_1$ of type $\tilde{D}_4$ and $\tilde{D}_4$. By \cite[Section 14]{KatsuraKondoMartin}, the fibration $f$ is quasi-elliptic and by \cite{Ito2}, we have $\MW(J(f)) = (\bbZ/2\bbZ)^2$.  A careful study of the fibrations $f'$ with fibers of type $\tilde{D}_8$ shows that $\MW(J(f))$ preserves all $(-2)$-curves, showing that $(\bbZ/2\bbZ)^2 \subseteq \Aut_{\nt}(S)$. To prove the claim, it suffices to show that these are all numerically trivial automorphisms on $S$.  Assume that there exists $g \in \Aut_{\nt}(S)$ acting non-trivially on $J(S)$. Then, by Table \ref{Table2}, the order of $g$ is not a power of $2$. On the other hand, consider the action of $g$ on the jacobian $J(f'):J(S)' \to \bbP^1$ of $f'$. Since $\MW(J(f')) = \bbZ/2\bbZ$, the action of $g$ on $J(S)'$ has order equal to a power of $2$. This is impossible by Table \ref{Table2}, so $g$ does not exist.

Finally, assume that $S$ is of type (D). Consider a quasi-elliptic fibration $f$ with double fibers $F_0$ and $F_1$ of type $\tilde{D}_6$ and $\tilde{A}_1^*$. By \cite{Ito2}, there is a $2$-torsion section of $J(f)$ that meets only the two fibers corresponding to the $F_i$ in a non-identity component. Let $g$ be the induced involution of $S$. Let $U \subseteq \Num(S)$ be the hyperbolic plane spanned by $F_0$ and the curve of cusps $C$ of $f$. Note that $U^\perp = E_8$. Let $E_0$ and $E_1$ be the two irreducible components of the reducible simple fiber of $f$. By construction, $g^*$ preserves the sublattice $L := \langle E_0 - E_1 \rangle \oplus A_1 \oplus D_6 \subseteq U^\perp = E_8$ spanned by $E_0 - E_1$ and the irreducible components of $F_0$ and $F_1$ which do not meet $C$. Since $g^*$ acts trivially on the first two summands, it acts trivially on the discriminant of $D_6$, so $g^* \in W(D_6)$. Since $D_6$ is spanned by effective curves and $g^*$ is induced by an automorphism of $S$, this shows that $g^*$ is trivial, hence $g \in \Aut_{\nt}(S)$. Finally, since $\Aut_{\ct}(J(S))$ is trivial by Table \ref{Table2}, Lemma \ref{translations} shows that $ \Aut_{\nt}(S) = \bbZ/2\bbZ$.
\end{proof}

\begin{remark}
Note that the proof of Theorem \ref{thm: Enriquesgroups} shows that every numerically trivial automorphism of a classical Enriques surface $S$ in characteristic $2$ is induced by translations by sections of the jacobian of some quasi-elliptic fibration on $S$.
\end{remark}

Finally, we can give the following characterization of Enriques surfaces with non-trivial $\Aut_{\nt}(S)$ in terms of the existence of certain quasi-elliptic fibrations on the Enriques surface. Assuming condition ${\rm (OS)}$ from Definition \ref{def: quasiellipticoggshafarevich}, this also gives a way of calculating the number of moduli.

\begin{corollary}
Let $S$ be a classical Enriques surface in characteristic $2$. Then, the following hold:
\begin{enumerate}
\item The group $\Aut_{\nt}(S)$ is non-trivial if and only if $S$ admits a quasi-elliptic fibration $f$ with double fibers as described in Table \ref{table: Enriquesmoduli}.
\item If $f$ satisfies condition ${\rm (OS)}$ of Definition \ref{def: quasiellipticoggshafarevich}, Table \ref{table: Enriquesmoduli} also gives the number of moduli of Enriques surfaces of a given type. In particular, assuming ${\rm (OS)}$, all types exist.
\end{enumerate}
\end{corollary}

\begin{table}[h]
 $
 \begin{array}{|c|c|c|}  \hline
 \text{ Type } & \text{ double fibers of } f &\text{ Number of moduli } \\ \hline  \hline
 (A) & 2\tilde{D}_8, 2\tilde{A}_0^{**} &2 \\  \hline
 (B) & 2\tilde{E}_7, 2\tilde{A}_1^{*} &1 \\  \hline
 (C) & 2\tilde{D}_4, 2\tilde{D}_4 & 2 \\  \hline
 (D) & 2\tilde{D}_6, 2\tilde{A}_1^* & 2 \\  \hline
 \end{array}
$
\bigskip
\caption{Moduli of classical Enriques surfaces with numerically automorphisms in characteristic $2$}
\label{table: Enriquesmoduli}
\end{table}

\begin{proof}
Let us first prove the statement about the quasi-elliptic fibrations. We have already seen in the proof of Theorem \ref{thm: Enriquesgroups} that Enriques surfaces of type (A), (B), (C), and (D) admit the stated genus one fibrations. Conversely, since the curve of cusps of a quasi-elliptic fibration $f$ on $S$ is a $(-2)$-curve which is a $2$-section of $f$, it is easy to see that an Enriques surface admitting one of the stated fibrations admits a configuration of $(-2)$-curves as in Theorem \ref{thm: Enriquesconfigs}.

To count the number of moduli, we recall that by Condition (OS), giving $f$ is the same as giving $J(f)$ and choosing a $2$-torsion point on each of the fibers of $J(f)$ that we want to be double fibers of $f$. By \cite{Ito2}, the number of moduli for $J(f)$ in the four cases is $0,0,1,$ and $0$, respectively. In Case (A), the choice of the cuspidal double fiber gives an additional parameter. Since the double fibers are additive, we get two additional parameters in each case for the choice of $2$-torsion point. Now, we use that $\Aut^0_{J(S)}$ is of dimension $1,1,1,$ and $0$, respectively, and that $\Aut^0_{J(S)}$ acts transitively on the set of cuspidal fibers of $J(f)$ in Case (A) and transitively on the set of $2$-torsion points on the identity component of the fiber of type $\tilde{E}_7,\tilde{D}_4,$ and $\tilde{D}_6$ in Cases (B), (C), and (D), respectively. Putting everything together yields the stated number of moduli.
\end{proof}

Since it is not known whether condition {\rm (OS)} always holds, we will describe explicit examples of each of these four types.
The explicit constructions of surfaces of Types $(A), (B),$ and $(C) $ can be found in \cite{KatsuraKondoMartin}. We will give here an explicit construction of surfaces of Type $(D)$. It follows a construction of surfaces of type (A) from \cite[Example 6.2.11]{DK}.

\begin{example}[Type (D)] Consider two quasi-elliptic fibrations defined by the pencils $|2F_1|$ and $|2F_2|$, where $F_1$ is the divisor of type $\tilde{D}_6$ described by the subdiagram containing the first  two vertical edges of the diagram for type $(D)$ and $|2F_2|$ has a simple fiber of type $\tilde{D}_8$ described by the subdiagram containing the first and the third vertical edges. We check that $F_1\cdot F_2 = 1$, so the linear system $|2F_1+2F_2|$ defines a bielliptic map $\phi:S\to \sfD_4$, where $\sfD_4$ is a 4-nodal anti-canonical del Pezzo surface of degree 4 in $\bbP^2$ (see \cite[Section 3.3]{CDL}). The morphism $\phi$ blows down $(-2)$-curves $C$ with $C.F_1 = C.F_2 = 0$. Checking the diagram, we see that there are precisely $8$ such curves, namely all curves contained in the fiber of type $\tilde{D}_8$ except for the curve of cusps of $|2F_1|$. It is known that the bielliptic map $\phi$ defines a birational model of $S$ isomorphic to the inseparable double plane 
$$w^2+xyF_6(x,y,z) = 0,$$
where $V(F_6(x,y,z))$ is a curve of degree 6 with a double point at the point $[1:0:0]$ and tacnodes at $[0:1:0], [0:0:1]$ with tangent directions $V(x)$ and $V(y)$. The singular points of the double plane lie over the zeros of the differential $dxyF_6$. We take 
$$F_6(x,y,z) = (x+y)(z^4(x+y)+az^3xy+z^2xy(bx+(b+1)y)+x^2y^2(x+y)),$$ 
where $a\ne 0$ and $b$ are parameters. The pencil $|2F_1|$ is equal to the pre-image of the pencil of conics $\lambda z^2+\mu xy = 0$ and the pencil $|2F_2|$ is equal to the pre-image of the pencil of lines $\lambda x+\mu y = 0$. We check that the singularity of the double plane over the point $(x:y:z)= (1:1:0)$ is formally isomorphic to the singular point $r^2+uv(u^2+v) = 0$. This is a rational double point of type $D_6^0$. It has also an ordinary double point over a point in the blow-up of the point $(0:0:1)$ and an ordinary double point at $(a:a:1)$. The pre-image of the line $x+y = 0$ that contains the point $(0:0:1), (1:1:0)$ and $(a:a:1)$ is the simple fiber of type $\tilde{D}_8$ of $|2F_2|$.  
The pre-image of the line $z= 0$ is a double fiber of type $\tilde{D}_6$ of the quasi-elliptic pencil $|2F_1|$. The pre-image of the exceptional curve of the blow-up of $\bbP^2$ at the point $(0:0:1)$ is a double fiber of type $\tilde{A}_1^*$ of the same fibration. The pre-image of the conic $a^2z+xy = 0$ is a simple fiber of type $\tilde{A}_1^*$ of $|2F_1|$. We leave it to the reader to verify that the minimal resolution of singularities of the double plane contains the configuration of $(-2)$-curves  Type (D).
In particular, by Theorem \ref{thm: Enriquesgroups}, we have $\Aut_{\nt}(S) = \bbZ/2\bbZ$. Using our explicit model, the numerically trivial involution can be described as
$x \leftrightarrow y$, $w \mapsto w + zxy(x+y)$.

\end{example}

%
%
%
%
%
%

 \end{document}